\documentclass[a4paper, 12pt]{article}

\usepackage{amsfonts, amsmath, amsthm}
\usepackage{amssymb}
\usepackage{latexsym}
\usepackage[dvips]{graphicx}
\usepackage{color}

\theoremstyle{plain}
\newtheorem{thm}{Theorem}[section]
\newtheorem{prp}{Proposition}[section]
\newtheorem{lem}{Lemma}[section]
\newtheorem{cor}{Corollary}[section]
\theoremstyle{definition}

\theoremstyle{remark}
\newtheorem{rmk}{Remark}[section]

\numberwithin{equation}{section}

\newcommand{\Z}{\mathbb{Z}}

\newcommand{\R}{\mathbb{R}}
\newcommand{\C}{\mathbb{C}}

\newcommand{\ene}[1]{{\mathcal{E}}_{#1}}

\newcommand{\cc}[1]{\overline{#1}}
\newcommand{\op}[1]{\mathcal{#1}}

\newcommand{\diag}{{\mathrm{diag}\,}}

\newcommand{\pa}{\partial}
\newcommand{\eps}{\varepsilon}
\newcommand{\jb}[1]{\langle #1 \rangle}
\newcommand{\Jb}[1]{\bigl\langle #1 \bigr\rangle}

\newcommand{\jbf}[1]{\left\langle #1 \right\rangle}

\newcommand{\sh}[1]{{#1}^{\sharp}}
\newcommand{\tm}{\tilde{m}}

\DeclareMathOperator{\realpart}{\rm Re}
\DeclareMathOperator{\imagpart}{\rm Im}

\newcommand{\dis}{\displaystyle}
\newcommand{\sumc}{\mathop{{\;\,\sum}'}}

\newcommand{\Co}[3]{C_{#1,#2,#3}}

\begin{document}
\title{
 Small data global existence for a class of \\
quadratic derivative nonlinear Schr\"odinger systems \\
in two space dimensions
 }

\author{
           Daisuke Sakoda\thanks{
              Department of Mathematics, Graduate School of Science, 
              Osaka University. 
              1-1 Machikaneyama-cho, Toyonaka, Osaka 560-0043, Japan. 
              (E-mail: {\tt d-sakoda@cr.math.sci.osaka-u.ac.jp})             }
           \and  
          Hideaki Sunagawa \thanks{
              Department of Mathematics, Graduate School of Science, 
              Osaka University. 
              1-1 Machikaneyama-cho, Toyonaka, Osaka 560-0043, Japan. 
              (E-mail: {\tt sunagawa@math.sci.osaka-u.ac.jp})
             }
}

\date{\today }   
\maketitle

\noindent{\bf Abstract:}\ 
This paper provides a small data global existence result for a class of quadratic 
derivative nonlinear Schr\"odinger systems in two space dimensions. 
This is an extension of the previous results by Li 
[Discrete Contin. Dyn. Syst., {\bf 32} (2012), 4265--4285] 
and Ikeda--Katayama--Sunagawa [Ann. H. Poincar\'e {\bf 16} (2015), 535--567].\\

\noindent{\bf Key Words:}\ 
Derivative nonlinear Schr\"odinger systems; null structure; gauge invariance; long-range interaction.\\

\noindent{\bf 2010 Mathematics Subject Classification:}\ 
35Q55, 35B40

\section{Introduction}  \label{sec_intro}
This paper deals with the global Cauchy problem for systems of nonlinear 
Schr\"odinger equations in two space dimensions. We consider systems of 
the type
\begin{equation}\label{NLS}
\left\{
\begin{array}{l}
\op{L}_{m_{j}} u_{j} =F_j(u, \pa_x u), \quad t>0,\, x\in\R^2, \, j=1,\dots,N, \\
u_j(0,x) = \varphi_j(x), \qquad x\in\R^2,\,j=1,\dots,N,
\end{array}
\right.
\end{equation}
where $\op{L}_{m_j} = i \pa_t + \frac{1}{2m_j}\Delta$, $i=\sqrt{-1}$, 
$\pa_t=\pa/\pa t$, $m_j \in \R \setminus \{ 0\}$ and 
$\Delta=\pa_{x_1}^2+\pa_{x_2}^2$ with $\pa_{x_a}=\pa/\pa x_a$ for 
$x=(x_a)_{a=1,\,2}\in\R^2$. 
$u=(u_j(t,\,x))_{1\leq j \leq N}$ is a $\C^N$-valued unknown function, and 
the nonlinear term $F=(F_j)_{1\leq j \leq N}$ is always assumed to be a 
quadratic homogeneous polynomial in $(u,\,\pa_x u,\,\cc{u},\,\cc{\pa_x u})$ 
where $\pa_x u=(\pa_{x_a} u_j(t,x))_{a=1,\,2;1\leq j \leq N}$. 
$\varphi=(\varphi_j(x))_{1\leq j \leq N}$ is a given $\C^N$-valued 
function which is assumed to be small in a suitable function space.  

Before going into the detail, let us recall some of the known results briefly. 
From the perturbative point of view, quadratic nonlinear Schr\"odinger 
equations in two space dimensions are of special interest because 
the best possible decay in $L_x^2$ of general quadratic nonlinear terms is 
$O(t^{-1})$, so the quadratic nonlinearity must be regarded as a 
long-range perturbation. In general, solutions develop singularities in 
finite time even if the initial data are sufficiently small and smooth 
(see e.g., Ikeda--Wakasugi \cite{IW} for an example on small data blow-up). 
Therefore we need some structural restrictions on the nonlinearity to obtain 
global solutions even for small data. 
Note that the critical exponent is expected to be $p=1+2/d$ 
when we consider the Schr\"odinger equations with the nonlinearity of degree 
$p$ in $d$-dimensional space. 
Roughly speaking, this exponent comes from the condition for 
convergence of the integral 
\[
 \int_1^{\infty} \frac{dt}{t^{\frac{d}{2}(p-1)}}.
\]
For the single equation case (i.e., $N=1$), small data global existence 
results for 2D quadratic NLS 
have been obtained by several papers (\cite{BHN}, \cite{Del}, \cite{GMS3}, 
\cite{GMS}, \cite{HN}, etc.), while less is known for the case of 
$N\ge 2$. An interesting feature in the system case is that 
large-time behavior of solutions is affected by not only the structure of 
the nonlinearity but also the combination of masses $(m_j)_{1\le j\le N}$, 
as discussed in recent works 
(\cite{ColCol}, \cite{HLN1}, \cite{HLN2}, \cite{HLO}, \cite{Hira}, \cite{IKS}, 
\cite{IKO}, \cite{KLS}, \cite{Kim}, \cite{Li}, \cite{LS1}, \cite{LS2}, 
\cite{OgaUri}, \cite{OzaSuna}, \cite{Uri} etc.).
A typical example of NLS system appearing in various physical settings is 
\begin{align}
\left\{
\begin{array}{l}
\op{L}_{m_{1}} u_{1} = \lambda_1 \cc{u_1} u_2,\\
\op{L}_{m_{2}} u_{2} = \lambda_2 {u_1}^2.  
\end{array}
\right.  \label{twoNLS}
\end{align}
In \cite{HLN1}, Hayashi--Li--Naumkin studied the Cauchy problem for 
this two-component system in detail. 
Roughly speaking, they proved small data global existence and 
time decay of the solution for \eqref{twoNLS} under the conditions 
\begin{align}
m_2=2m_1
\label{mass_res2}
\end{align}
and
\begin{align}
\realpart(\lambda_1\lambda_2)>0, \quad 
\imagpart(\lambda_1\lambda_2)=0.
 \label{structure2}
\end{align}
(A closely related result on quadratic nonlinear Klein-Gordon systems 
in $\R^2$ can be found in \cite{KawaSuna2}; 
see also \cite{DFX}, \cite{KOS}.) 
Soon later this result was generalized by Li~\cite{Li} to more general 
systems, but it should be noted that the approach of \cite{Li} 
are available only for the case where the nonlinear term does not 
contain the derivatives of $u$ (i.e., $F_j=F_j(u)$) because the presence of 
derivatives in the nonlinearity causes a derivative loss in general. 
In \cite{IKS}, Ikeda--Katayama--Sunagawa studied a derivative nonlinear 
case and found a kind of {\em null structure} in it 
(cf. \cite{Kla}, \cite{Chr}, \cite{Hor}). 
To be more precise, they considered the three-component system with the 
nonlinearity given by 
\begin{align}
\left\{\begin{array}{l}
 \dis{F_1(u,\pa_x  u) =
  \sum_{|\alpha|,|\beta|\le 1} \Co{1}{\alpha}{\beta}  \, 
  (\overline{\pa^{\alpha} u_2})  (\pa^{\beta} u_3),
 }\\[7mm]
 \dis{F_2(u,\pa_x u) =
  \sum_{|\alpha|,|\beta|\le 1} \Co{2}{\alpha}{\beta}  \, 
   (\pa^{\alpha} u_3) (\overline{\pa^{\beta} u_1}),  
 }\\[7mm]
 \dis{F_3(u,\pa_x u) =
  \sum_{|\alpha|,|\beta|\le 1} \Co{3}{\alpha}{\beta}  \, 
  (\pa^{\alpha} u_1) (\pa^{\beta} u_2)
 }
\end{array}\right.
 \label{Nonlin}
\end{align}
with complex constants $\Co{k}{\alpha}{\beta}$, and obtained 
small data global existence and scattering result 
under the conditions 
\begin{align}
m_1+m_2=m_3
\label{mass_res3}
\end{align} 
and 
\begin{align}
\Lambda_1(\xi)=\Lambda_2(\xi)=\Lambda_3(\xi)=0,
\quad \xi \in \R^2, 
 \label{null_con3}
\end{align}
where 
\begin{align*}
\left\{\begin{array}{l}
 \dis{\Lambda_1(\xi) =
 \sum_{|\alpha|,|\beta|\le 1} \Co{1}{\alpha}{\beta}  \, 
 (\overline{im_2 \xi})^{\alpha}\, (im_3\xi)^{\beta},} 
 \\[7mm]
 \dis{\Lambda_2(\xi) =
 \sum_{|\alpha|,|\beta|\le 1} \Co{2}{\alpha}{\beta}  \, 
 (im_3 \xi)^{\alpha} \, (\overline{im_1\xi})^{\beta},}
 \\[7mm]
 \dis{\Lambda_3(\xi) =
 \sum_{|\alpha|,|\beta|\le 1} \Co{3}{\alpha}{\beta}  \, 
 (im_1\xi)^{\alpha} \,  (im_2\xi)^{\beta}.}
\end{array}\right.
\end{align*}
We refer the readers to Section~4 of \cite{IKS} for a characterization of 
the nonlinearity satisfying \eqref{null_con3} in terms of 
special quadratic forms called {\em{the null gauge forms}} 
and {\em{the strong null forms}}. 
We note that the two-component system \eqref{twoNLS} can be viewed as a 
degenerate case of the three-component system with 
\[
F_1=\lambda_1\cc{u_2}u_3,\ \ F_2=\lambda_2u_3\cc{u_1}, \ \ F_3=\lambda_3u_1u_2,
\] 
and the condition \eqref{mass_res2} for \eqref{twoNLS} can be 
interpreted as \eqref{mass_res3} for this extended system. 
However, this system fails to satisfy \eqref{null_con3} 
unless $\lambda_1=\lambda_2=\lambda_3=0$, so 
the result of \cite{IKS} does not cover that of \cite{Li}. 

The aim of the present paper is to extend  and unify the results 
of \cite{Li} and \cite{IKS}. More precisely, we will introduce a new 
structural condition on $(F_j)_{1\le j \le N}$ and $(m_j)_{1\le j\le N}$ 
under which the small data solution exists globally in time and decays at the 
rate $O(t^{-1})$ as $t\to \infty$ in $L^{\infty}$. 
Another novelty of our result is that it is applicable to the system 
introduced by Colin--Colin \cite{ColCol} (see \eqref{CC_model} below), 
which has not been covered in the previous works.

\section{Main Results}  \label{sec_result}
In the subsequent sections, we will use the following notations: 
We set $I_N=\{1,\ldots, N\}$ and 
$\sh{I}_N =\{1,\ldots, N, N+1,\ldots, 2N \}$. 
For $z =(z_j)_{j\in I_N}\in \C^N$, we write 
\[
 \sh{z} =(\sh{z}_k)_{k \in \sh{I}_N}
 :=(z_1,\ldots,z_N, \overline{z_1},\ldots, \overline{z_N})
\in \C^{2N}.
\]
Then general quadratic nonlinear term $F=(F_j)_{j \in I_N}$ can be written as 
\begin{align}\label{nonlin}
 F_j(u,\pa_x u)
 =\sum_{|\alpha|, |\beta|\le 1}\sum_{k, l \in \sh{I}_N}
  C_{j, k, l}^{\alpha,\beta}  (\pa_x^{\alpha} \sh{u}_{k}) 
 (\pa_x^{\beta} \sh{u}_{l}) 
\end{align}
with suitable $C_{j, k, l}^{\alpha,\beta}  \in \C$. 
With this expression of $F$, we define 
$p=(p_j(\xi;Y))_{j \in I_N}:\R^2 \times \C^N \to \C^N$ by  
\begin{align*}
 p_j(\xi; Y)
 :=
 \sum_{|\alpha|, |\beta|\le 1}\sum_{k, l \in \sh{I}_N}
 C_{j, k, l}^{\alpha, \beta} (i\tm_{k} \xi)^{\alpha}
 (i\tm_{l} \xi)^{\beta}  \sh{Y}_{k} \sh{Y}_{l} 
\end{align*}
for $\xi =(\xi_1,\xi_2)\in \R^2$ and $Y=(Y_j)_{j \in I_N} \in \C^N$, 
where 
\[
 \tm_k=\left\{\begin{array}{cl}
 m_k& (k=1,\ldots, N),\\[4mm]
 -m_{(k-N)} & (k = N+1,\ldots, 2N).
 \end{array}\right.
\]
In what follows, we denote by $\jb{\cdot, \cdot}_{\C^N}$ the standard scalar 
product in $\C^N$, i.e., 
$$
 \jb{z,w}_{\C^N}=\sum_{j=1}^{N}  z_j \overline{w_j}
$$
for $z=(z_j)_{j \in I_N}$ and $w=(w_j)_{j \in I_N} \in \C^N$. 
For $s, \sigma \in \Z_+:=\{0,1,2,\ldots\}$, 
we denote by $H^s$ the $L^2$-based Sobolev space of order $s$, 
and the weighted Sobolev space $H^{s,\sigma}$ is defined by 
$\{\phi \in L^2\, |\, \jb{x}^{\sigma} \phi \in H^s \}$ equipped with 
the norm $\|\phi\|_{H^{s,\sigma}}=\|\jb{\,\cdot\,}^{\sigma} \phi\|_{H^s}$, 
where $\jb{x}=\sqrt{1+|x|^2}$. 

The main result is as follows:

\begin{thm} \label{thm_sdge}
Assume the following two conditions {\rm ({\bf a})} and {\rm ({\bf b})} are 
satisfied:
\begin{description}
 \item[{\rm ({\bf a})}]
For each $j \in I_N$ and $k, l \in \sh{I}_N$,
\[
m_j \neq \tilde{m}_{k} + \tilde{m}_{l}\  
\text{implies} \  
C_{j,k,l}^{\alpha,\beta}=0 \ \text{ for}\ 
\alpha, \beta\in \Z_+^2 \ \text{with}\ |\alpha|,|\beta|\leq 1.
\]
 \item[{\rm ({\bf b})}]
There exists an $N\times N$ positive Hermitian matrix $H$ such that
\[\imagpart\langle p(\xi;Y),HY\rangle_{\C^N}=0\]
for $(\xi,Y)\in\R^2\times\C^N$.
\end{description}
Let $\varphi \in \bigcap_{k=0}^5 H^{11-k, k}(\R^2)$ 
and suppose that  $\eps := \sum_{k=0}^5 \|\varphi \|_{H^{11-k,k}(\R^2)}$ 
is sufficiently small. 
Then \eqref{NLS} admits a unique global solution 
$u\in C([0,\infty);\bigcap_{k=0}^5 H^{11-k, k}(\R^2))$. 
Moreover we have
\begin{equation}
\|u(t,\cdot)\|_{L^\infty(\R^2)}\leq \frac{C\eps}{1+t},\qquad \|u(t,\cdot)\|_{L^2(\R^2)}
\leq C\eps
\end{equation}
for $t\geq 0$, where $C$ is a positive constant not depending on $\eps$.
\end{thm}

\begin{rmk}\label{rmk1}
Analogous result for 1D cubic case  has been obtained in the previous 
work by Li--Sunagawa \cite{LS1}. 
Remember that 1D cubic case is another critical situation, 
that is, $3=(1+2/d)|_{d=1}$. 
However, we need several modifications to prove Theorem~\ref{thm_sdge} 
because the approach of \cite{LS1} 
relies heavily on one-dimensional nature.
Another remark concerning this point is 
that the condition ({\bf b}) above can be replaced by the following 
apparently weaker one: 
\begin{description}
 \item[{\rm ({\bf b}$'$)}]
There exists an $N\times N$ positive Hermitian matrix $H$ such that
\[\imagpart\langle p(\xi;Y),HY\rangle_{\C^N}\le 0\]
for $(\xi,Y)\in\R^2\times\C^N$.
\end{description}
Indeed, since $Y\mapsto \imagpart\langle p(\xi;Y),HY\rangle_{\C^N}$ is an odd 
function, we can see that ({\bf b}$'$) yields ({\bf b}) 
by substituting $-Y$ in place of  $Y$ in ({\bf b}$'$). It is worth noting 
that this equivalence fails if the original nonlinearity is cubic. 
For closely related works on the wave equation case, 
see \cite{KMatoS}, \cite{KMatsS} and Chapter~10 of \cite{Kata}. 
\end{rmk}

\begin{rmk}\label{rmk2}
If $\pa_x u$ is not included in the nonlinear term, then 
the conditions ({\bf a}) and ({\bf b}) are essentially the same as ones given 
in \cite{Li}. 
In particular, when we focus on the two-component system \eqref{twoNLS}, 
we can see that \eqref{mass_res2} plays the role of ({\bf a}) and that 
\eqref{structure2} leads to  ({\bf b}) with 
\[
 H=\begin{pmatrix} \lambda_1\lambda_2& 0\\ 0& |\lambda_1|^2\end{pmatrix}.
\]
\end{rmk}

\begin{rmk}\label{rmk3}
 If $p(\xi;Y)$ vanishes identically on $\R^2\times \C^N$, then the condition 
({\bf b}) is trivially satisfied.
Therefore our result can be viewed as an extension of \cite{IKS}. 
Under this stronger condition, we can show also that 
the solution $u(t)$ to \eqref{NLS} is asymptotically free 
by the same method as in \cite{IKS}. Note that ({\bf b}) does not imply 
the asymptotically free behavior in general, 
because non-existence of asymptotically free solutions for \eqref{twoNLS} 
has been shown in \cite{HLN1}.
\end{rmk}

\begin{rmk}\label{rmk4} 
In \cite{ColCol}, Colin--Colin introduced the following system 
as a model of laser-plasma interaction: 
\begin{align}
\begin{cases}
  i\pa_t A_C+\alpha \Delta A_C = -(\nabla\cdot E)A_R,\\
  i\pa_t A_R+\beta \Delta A_R = -(\nabla\cdot \cc{E})A_C,\\
  i\pa_t E+\gamma \Delta E = \nabla(\cc{A_C} \cdot A_R),
\end{cases}
\label{CC_model}
\end{align}
where $A_C$, $A_R$, $E$ are $\C^2$-valued functions and $\alpha$, $\beta$, 
$\gamma$ are non-zero real constants. 
When we define $u=(u_j)_{1\le j \le 6}$ by
\[
 A_C=\begin{pmatrix}u_1\\ u_2\end{pmatrix}, \ \ 
 A_R=\begin{pmatrix}u_3\\ u_4\end{pmatrix}, \ \ 
 E=\begin{pmatrix}u_5\\ u_6\end{pmatrix}
\]
and set $m_1=m_2=1/(2\alpha)$, $m_3=m_4=1/(2\beta)$, $m_5=m_6=1/(2\gamma)$, 
we have the six-component NLS system with 
\begin{align*}
\begin{cases}
 F_1=-(\pa_1 u_5+\pa_2 u_6)u_3,\\
 F_2=-(\pa_1 u_5+\pa_2 u_6)u_4,\\
 F_3=-(\pa_1 \cc{u_5} + \pa_2 \cc{u_6})u_1,\\
 F_4=-(\pa_1 \cc{u_5} + \pa_2 \cc{u_6})u_2,\\
 F_5=\pa_1 (u_1 \cc{u_3} + u_2 \cc{u_4}),\\
 F_6=\pa_2 (u_1 \cc{u_3} + u_2 \cc{u_4}).
\end{cases}
\end{align*}
For this system, we can easily see that 
({\bf a}) is satisfied if 
\begin{align}
\frac{1}{\alpha} = \frac{1}{\beta} + \frac{1}{\gamma}. 
\label{resonance}
\end{align} 
The condition ({\bf b}) is also satisfied with  $H=\diag (2,2,1,1,1,1)$ 
if we assume \eqref{resonance}. Indeed,  since 
\begin{align*}
\begin{cases}
 p_1(\xi; Y)=-i(m_5 \xi_1 Y_5 + m_6\xi_2 Y_6)Y_3,\\
 p_2(\xi; Y)=-i(m_5\xi_1 Y_5+ m_6\xi_2 Y_6)Y_4,\\
 p_3(\xi; Y)=i(m_5\xi_1 \cc{Y_5} + m_6\xi_2 \cc{Y_6})Y_1,\\
 p_4(\xi; Y)=i(m_5\xi_1 \cc{Y_5} + m_6\xi_2 \cc{Y_6})Y_2,\\
 p_5(\xi; Y)=i\xi_1 \{(m_1-m_3)Y_1 \cc{Y_3} + (m_2-m_4)Y_2 \cc{Y_4} \},\\
 p_6(\xi; Y)=i\xi_2 \{(m_1-m_3)Y_1 \cc{Y_3} + (m_2-m_4)Y_2 \cc{Y_4}\},
\end{cases}
\end{align*}
we have 
\begin{align*}
 \imagpart\jb{p(\xi; Y), HY}_{\C^6}
=&
\frac{1}{2}\left(\frac{1}{\alpha}-\frac{1}{\beta}-\frac{1}{\gamma}\right)
\realpart\Bigl[ (Y_1\cc{Y_3}+Y_2\cc{Y_4})(\xi_1 \cc{Y_5} +\xi_2\cc{Y_6})\Bigr]
\end{align*}
which vanishes identically on $\R^2\times \C^6$ under the relation 
\eqref{resonance}.
\\
\end{rmk}

Now, for the convenience of the readers, let us 
give a heuristic explanation for the roles played by 
our conditions ({\bf a}) and ({\bf b}). As in \cite{IKS} and \cite{LS1}, 
our starting point is to recall the fact that, if $u_j^0$ solves 
$\mathcal{L}_{m_j} u_j^0=0$ with $u^0_j(0,x)=\varphi_j(x)$, it holds that 
\[
 \pa_x^s u_j^0(t,x)\sim \left(im_j \frac{x}{t}\right)^s \frac{m_j}{it} 
 \hat{\varphi}_j\left( \frac{m_j x}{t}\right) e^{i\frac{m_j |x|^2}{2t}}
 +\cdots
\]
as $t\to+\infty$. Viewing it as a rough approximation of the solution 
$u_j$ for \eqref{NLS}, we may expect that $\pa_x^s u_j(t,x)$ could be 
better approximated by 
\[
 \left(im_j \frac{x}{t}\right)^s \frac{1}{t} 
 Y_j\left(\log t,  \frac{x}{t}\right) e^{i\frac{m_j |x|^2}{2t}}
\]
with a suitable function $Y=(Y_j(\tau,\xi))_{j\in I_N}$, 
where $\tau=\log t$, $\xi=x/t$ and $t\gg 1$. 
Note that 
$Y_j(0,\xi)=-im_j\, \hat{\varphi}_j(m_j \xi)$ 
and that the extra variable $\tau=\log t$ is responsible 
for possible long-range nonlinear effect. 
Substituting the above expression into \eqref{NLS} 
and keeping only the leading terms, 
we can see (at least formally) that $Y_j$ should satisfy 
the ordinary differential equation
\begin{align}
 i\pa_{\tau} Y_j(\tau,\xi) =p_j(\xi; Y(\tau,\xi))
\label{profile_ODE}
\end{align}
(where $\xi$ is regarded as a parameter) under the condition ({\bf a}). 
We remark that ({\bf a}) implies the symmetry 
\[
(u_j)_{j \in I_N}\mapsto (e^{im_j \theta} u_j)_{j \in I_N},\quad 
\theta \in \R,
\]
in \eqref{NLS}. This property, which we call {\em the gauge invariance},  
is used in this step. Another structure comes into play when the gauge 
invariance is violated (see \cite{LS2} for a detailed study on this issue 
in 1D cubic case). 
Next let $H$ be a positive Hermitian matrix. Then 
\eqref{profile_ODE} yields  
\[
 \pa_{\tau} \jb{Y(\tau,\xi),HY(\tau,\xi)}_{\C^N} 
 = 
 2\imagpart 
 \jb{p(\xi;Y(\tau,\xi)), HY(\tau,\xi)}_{\C^N},
\]
and the condition ({\bf b}), or equivalently ({\bf b}$'$), is just what makes this 
quantity non-positive. 
Since $|Y|^2$ and $\jb{Y,HY}_{\C^N}$ are equivalent, 
the inequality $\pa_{\tau}\jb{Y,HY}_{\C^N}\le 0$ 
implies that $Y(\tau,\xi)$ remains bounded when $\tau$ becomes large. 
Going back to the original variables, we see that the solution 
$u(t,x)$ for \eqref{NLS} decays like $O(t^{-1})$ in $L^{\infty}_x$ as 
$t\to +\infty$ under ({\bf b}). This is a heuristic reason why 
the solution has a desired decay property under ({\bf a}) and ({\bf b}).
We remark that \eqref{profile_ODE} is reduced to the trivial equation 
if we assume the stronger condition that $p(\xi;Y)$ vanishes identically. 
This gives a heuristic reason why the solution should be free from the 
long-range effect under this stronger condition, as mentioned in 
Remark~\ref{rmk3}.

Our strategy of the proof of Theorems \ref{thm_sdge} is to justify the above 
heuristic argument. Let us give a more detailed summary of our approach. 
The key is to introduce 
\[
 A_j(t,\xi)
 =
 \mathcal{F}_{m_j}[\mathcal{U}_{m_j}(t)^{-1} u_j(t,\cdot)](\xi),
\] 
where $\mathcal{F}_m$ and $\mathcal{U}_m(t)$ are given in 
Section~\ref{sec_prelim} below. 
Roughly speaking, this $A_j(t,\xi)$ is expected to play the role of 
$Y_j(\log t, \xi)$. We will see in Section~\ref{sec_pf_a_priori} that 
$A=(A_j(t,\xi))_{j\in I_N}$ satisfies
\begin{align}
 i\pa_t A_j(t,\xi)=\frac{1}{t}p_j(\xi; A(t,\xi))+ O(t^{-3/2+2\delta}) 
\label{reduced_sys}
\end{align}
and 
\[
 \|u(t,\cdot)\|_{L_x^{\infty}} 
 \le 
 t^{-1}\|A(t,\cdot)\|_{L_{\xi}^{\infty}} + O(t^{-3/2+\delta})
\]
with $0<\delta < 1/4$. 
To control the remainder terms, we need several 
$L^2$-estimates involving the operator $\op{J}_{m}$.
In the 1D cubic case, only one action of $\op{J}_{m}$ is enough 
for getting desired a priori $L^2$-bounds. 
This is the point where the one-dimensional nature 
(such as the imbedding $H^1(\R^1) \hookrightarrow L^{\infty}(\R^1)$) is 
used in \cite{LS1}. 
However, since we are considering the problem in $\R^2$ now, 
we have to use $\op{J}_{m}$ several times. 
Then, through the relation $\op{F}_m\op{U}_m^{-1}\op{J}_{m,a}
=({i}/{m})\pa_{\xi_a}\op{F}_m\op{U}_m^{-1}$, 
we must differentiate \eqref{reduced_sys} with respect to $\xi$ several times, 
and it destroys the good structure coming from the condition ({\bf b}). 
We will overcome this difficulty by getting suitable pointwise bounds 
for $\jb{\xi}^{8-|\gamma|}\pa_\xi^{\gamma} A_j(t,\xi)$ 
($|\gamma|\le 3$) 
up to moderate growth in $t$.   
This is the new ingredient of our proof.

\section{Preliminaries}  \label{sec_prelim}

This section is devoted to preliminaries on useful identities and estimates related to 
the operator $\op{J}_m$ and the free evolution group $\op{U}_m$, and 
on energy inequalities associated with the (sesqui-)linearized system.
In what follows we will denote several positive constants by the same letter 
$C$, which may vary from one line to another.

\subsection{The operator $\op{J}_m$ and the free evolution group $\op{U}_m$}

We set $\op{L}_m=i\pa_t+\dfrac{1}{2m}\Delta$, 
$\op{J}_{m}(t)=(\op{J}_{m,\,a}(t))_{a=1,\,2}$, 
$\op{J}_{m,\,a}(t)=x_a+\dfrac{i t}{m}\pa_{x_{a}}$ 
for $m\in\R\setminus\{0\}$. 
For simplicity of notation, we often write  $\op{J}_{m,a}$ 
instead of $\op{J}_{m,a}(t)$. It is easy to check that 
\begin{equation}\label{commute}
[\op{L}_m,\pa_{x_a}]=[\op{L}_{m},\op{J}_{m,a}]=0,
\quad[\op{J}_{m,a},\pa_{x_b}]=-\delta_{ab},
\quad[\op{J}_{m,a},\op{J}_{m,b}]=0~~(a,b=1,2),
\end{equation}
where $[\cdot,\cdot]$ denotes the commutator of two linear operators and 
$\delta_{ab}=1$ (if $a=b$), $=0$ (if $a \neq b$). 
We write $\op{J}_m^\alpha=\op{J}_{m,1}^{\alpha_1}\op{J}_{m,2}^{\alpha_2}$ 
for a multi-index $\alpha=(\alpha_1,\alpha_2)\in\Z_+^2$. 
The following identity is useful:
\begin{equation}\label{Jhenkei1}
\op{J}_{m,a}f
=\dfrac{i t}{m}e^{i m\theta}\pa_{x_a}(e^{-i m\theta}f),
\quad \theta=\frac{|x|^2}{2t}.
\end{equation}
Indeed, we can deduce the following lemmas from \eqref{Jhenkei1} and 
\eqref{commute}:
\begin{lem}\label{leibniz}
Let $m$, $\mu_1,\,\mu_2$ be non-zero real constants satisfying 
$m=\mu_1+\mu_2$. We have
\begin{align*}
\op{J}_{m,a}(\phi_1 \phi_2)
&=
\dfrac{\mu_1}{m} \left(\op{J}_{\mu_1,a}\phi_1\right) \phi_2
+\dfrac{\mu_2}{m} \phi_1 \left(\op{J}_{\mu_2,a} \phi_2\right),
\\
\op{J}_{m,a}(\phi_1 \cc{\phi_2})
&=
\dfrac{\mu_1}{m} \left(\op{J}_{\mu_1,a}\phi_1\right) \cc{\phi_2}
+\dfrac{\mu_2}{m} \phi_1 \left(\cc{\op{J}_{-\mu_2,a} \phi_2}\right),
\\
\op{J}_{m,a}(\cc{\phi_1}\cc{\phi_2})
&=
\dfrac{\mu_1}{m} \left(\cc{\op{J}_{-\mu_1,a}\phi_1}\right) \cc{\phi_2}
+\dfrac{\mu_2}{m} \cc{\phi_1} \left(\cc{\op{J}_{-\mu_2,a}\phi_2}\right)
\end{align*}
for $a=1$, $2$ and smooth $\C$-valued functions $\phi_1$, $\phi_2$.
\end{lem}

\begin{lem}\label{lem2}
Let $m$ be a non-zero real constant. We have 
\[ 
|\op{J}_{m}^{\beta}\pa_{x}^{\alpha}\phi| 
\leq C \sum_{\substack{\beta^{\prime} \leq \beta \\ 
\alpha^{\prime}\leq \alpha}}|\pa_{x}^{\alpha^{\prime}}
\op{J}_{m}^{\beta^{\prime}}\phi|
\]
for $\alpha,\ \beta\in \Z_+^2$ and a smooth function $\phi$.
\end{lem}

Next we set $\op{U}_{m}(t):=e^{i\frac{t}{2m}\Delta}$, that is, 
\begin{equation*}
\left(\op{U}_{m}(t)\phi\right)(x)
=\dfrac{m}{2\pi i t }\int_{\R^2}e^{i m\frac{|x-y|^2}{2t}}\phi(y) dy
\end{equation*}
for $m\in\R\setminus\{0\}$, and $t>0$.
Then we have
\begin{equation}\label{U_m}
\op{U}_{m}[x_a\phi]
=\op{J}_{m,\,a}\op{U}_m\phi,
\quad
\op{U}_m\pa_{x_b}\phi=\pa_{x_b}\op{U}_m\phi.
\end{equation}
We also introduce the scaled Fourier transform $\op{F}_m$ by
\begin{equation*}
\bigl(\op{F}_{m}\phi\bigr)(\xi)
:=-i m\hat\phi(m\xi)=\frac{m}{2\pi i}\int_{\R^2}e^{-i my\cdot\xi}\phi(y)dy,
\end{equation*}
as well as auxiliary operators
\begin{equation*}
(\op{M}_{m}(t)\phi)(x)
:=e^{i m\frac{|x|^2}{2t}}\phi(x),\ \ \ 
(\op{D}(t)\phi)(x):=\dfrac{1}{t}\phi\left(\frac{x}{t}\right),\ \ \ 
\op{W}_{m}(t)\phi:=\op{F}_m\op{M}_m(t)\op{F}^{-1}_m\phi,
\end{equation*}
where $\hat\phi$ denotes the standard Fourier transform of $\phi$, i.e.
\[
\hat\phi(\xi)=\bigl(\op{F}\phi \bigr)(\xi)
:=\dfrac{1}{2\pi}\int_{\R^2}e^{-i y\cdot\xi}\phi(y)dy.
\]
Then we can see that
\begin{equation}\label{F_m}
\op{F}_m[x_a\phi](\xi)
=\frac{i}{m}\pa_{\xi_a}\bigl(\op{F}_m\phi\bigr)(\xi),
\quad 
\op{F}_m\pa_{x_b}\phi=i m\xi_b\op{F}_{m}\phi.
\end{equation}
and that $\op{U}_m$ can be decomposed into
\begin{equation*}
\op{U}_{m}=\op{M}_m\op{D}\op{F}_m\op{M}_m=\op{M}_m\op{D}\op{W}_m\op{F}_m.
\end{equation*}
Note that the operators $\op{U}_m,\ \op{F}_m,\ \op{M}_m,\ \op{D}$ and 
$\op{W}_m$ above are isometries on $L^2$. 
By \eqref{U_m} and \eqref{F_m}, we can easily check that
\begin{equation}\label{lem1}
(i m\xi)^{\alpha}\op{F}_{m}\op{U}^{-1}_{m}\phi
=\op{F}_{m}\op{U}^{-1}_{m}\pa_x^{\alpha}\phi, \quad
\left(\frac{i}{ m}\pa_{\xi}\right)^{\beta}\op{F}_{m}\op{U}^{-1}_{m}\phi
=\op{F}_{m}\op{U}^{-1}_{m}\op{J}_{m}^{\beta}\phi.
\end{equation}
for all $\alpha,\,\beta\in\Z_+^2$.

\begin{lem}\label{lem3} 
Let $m$ be a non-zero real constant. 
We set 
$A (t,\xi)=\op{F}_{m} \bigl[\op{U}^{-1}_{m}(t)\phi(t,\cdot)\bigr](\xi)$ 
for a smooth function $\phi(t,x)$.
\begin{itemize}
\item[(1)]
For $s, \sigma\in \Z_+$, we have
\[
 \|A(t,\cdot)\|_{H^{s,\sigma}}
 \le 
 C \sum_{|\beta|\le s} \|\op{J}_{m}(t)^{\beta}\phi(t,\cdot)\|_{H^{\sigma}}.\]
\item[(2)]
For $\alpha,\ \beta\in\Z_+^2$, we have
\[
\|\pa_x^\alpha \op{J}_m(t)^\beta \phi(t,\cdot)\|_{L_x^2} 
\le 
C \|\jb{\,\cdot\,}^{|\alpha|+2}\pa_{\xi}^\beta A(t,\cdot) \|_{L_{\xi}^\infty}.
\]
\end{itemize}
\end{lem}

\begin{proof}
By Lemma~\ref{lem2} and \eqref{lem1}, we have 
\[
\|A\|_{H^{s,\sigma}} 
\leq 
C\sum_{\substack{ |\alpha|\le \sigma\\ |\beta|\le s}}
\|\pa_{\xi}^{\beta}(\xi^{\alpha} A)\|_{L^2}
\leq 
C\sum_{\substack{ |\alpha|\le \sigma\\ |\beta|\le s}}
\|\op{F}_m\op{U}_m^{-1}\op{J}_{m}^{\beta}\pa_x^{\alpha}\phi\|_{L^2} 
 \le 
 C \sum_{|\beta|\le s} \|\op{J}_{m}^{\beta}\phi(t,\cdot)\|_{H^{\sigma}}
\]
and
\begin{align*}
\|\pa_x^\alpha \op{J}_m^\beta \phi\|_{L^2_x}
\le 
C\|\jb{\xi}^{|\alpha|}\pa_{\xi}^\beta 
\op{F}_m\op{U}_m^{-1}\phi \|_{L^2_\xi}
\le 
C\|\jb{\xi}^{-2}\|_{L^2_\xi}
\|\jb{\xi}^{|\alpha|+2} \pa_{\xi}^\beta A \|_{L_\xi^\infty},
\end{align*}
as desired.
\end{proof}

\begin{lem}\label{W}
We have
\[
\|\op{W}_m(t)\phi-\phi\|_{L^{\infty}}
+
\|\op{W}_m(t)^{-1}\phi-\phi\|_{L^{\infty}}
\le 
Ct^{-1/2}\|\phi\|_{H^2}
\]
for $t>0$.
\end{lem}

\begin{proof}
From the inequalities $|e^{i\theta}-1|\leq |\theta|$ and 
$\|\phi\|_{L^\infty}\leq C\|\phi\|_{L^2}^{1/2}\|\Delta \phi \|_{L^2}^{1/2}$, 
we see that
\begin{align}
\|\left(\op{W}_{m}^{\pm 1}-1\right)\phi \|_{L^\infty} 
&\leq  
C\|\left(\op{M}_m^{\pm 1}-1\right)\op{F}_m^{-1}\phi\|_{L^2}^{1/2}\ 
\|\Delta\left(\op{W}_{m}^{\pm 1}-1\right) \phi \|_{L^2}^{1/2} \notag \\
&\leq 
C\|t^{-1}|x|^2\op{F}_m^{-1}\phi\|_{L^2}^{1/2}\ 
\|\left(\op{W}_{m}^{\pm 1}-1\right) \Delta\phi \|_{L^2}^{1/2}\notag \\
&\leq 
Ct^{-1/2}\| |x|^2\op{F}_m^{-1}\phi\|_{L^2}^{1/2}\ \|\Delta\phi \|_{L^2}^{1/2} 
\notag \\
&\leq 
Ct^{-1/2}\|\phi\|_{H^2}.\notag
\end{align}
\end{proof}

\begin{lem}\label{Linftylem}
Let $m\in \R\setminus \{0\}$. We have
\[
\| \phi- \op{M}_{m}\op{D}\op{F}_{m}\op{U}^{-1}_{m}\phi \|_{L^\infty} 
\leq 
Ct^{-3/2}\sum_{|\beta|\leq2}\|\op{J}_m^\beta\phi\|_{L^2}
\]
and
\begin{align}
&\|\phi\|_{L^\infty} 
\leq 
t^{-1}\|\op{F}_m\op{U}^{-1}_m\phi\|_{L^\infty} 
+ Ct^{-3/2}\sum_{|\beta|\leq2}\|\op{J}_m^\beta \phi \|_{L^2} 
\label{Linfty}
\end{align}
for $t>0$.
\end{lem}

\begin{proof}
By Lemmas \ref{lem3}, \ref{W} and the relation 
$\op{U}_m=\op{M}_m\op{D}\op{W}_m\op{F}_m$, we have
\begin{align*}
\| \phi- \op{M}_{m}\op{D}\op{F}_{m}\op{U}^{-1}_{m}\phi \|_{L^\infty} 
&= 
\|\op{M}_{m}\op{D}\left(\op{W}_m -1\right)\op{F}_{m}\op{U}^{-1}_{m}\phi 
\|_{L^\infty} \\
&\leq 
t^{-1} \| \left(\op{W}_m -1\right)\op{F}_{m}\op{U}^{-1}_{m}\phi \|_{L^\infty} 
\\
&\leq 
Ct^{-3/2} \|\op{F}_{m}\op{U}^{-1}_{m}\phi \|_{H^2} \\
&\leq 
Ct^{-3/2} \sum_{|\beta|\leq2}\|\op{J}_m^\beta\phi\|_{L^2}.
\end{align*}
The second inequality follows immediately from the first one.
\end{proof}

\subsection{Energy inequalities}  \label{sec_smoothing}

In this subsection we focus on the Cauchy problem for 
\begin{equation}\label{smoothv}
\op{L}_{m_j}v_j =\sum_{k\in \sh{I}_N}\sum_{a=1}^2g_{jk,a}\pa_{x_a}\sh{v}_k
+G_j,\quad (t,x)\in(0,T)\times\R^2, \ j\in I_N,
\end{equation} 
where $m_j\in\R\setminus\{0\}$, $T>0$ are constants, and $g=(g_{jk,a})$, 
$G=(G_j)$ are 
given functions of $(t,x)$ having suitable regularity and decay at spatial 
infinity. Our goal here is to derive an $L^2$-bound for the solution 
$v=(v_j)_{j\in I_N}$ to this system, 
keeping in mind applications to \eqref{NLS} in the subsequent sections. 
If $g_{jk,a}(t,x)\equiv 0$, there is no difficulty because the standard energy 
integral method immediately yields
\[
 \|v(t,\cdot)\|_{L^2}
 \le 
 \|v(t_0,\cdot)\|_{L^2} +\int_{t_0}^t \|G(\tau,\cdot)\|_{L^2}\, d\tau
\] 
for $t_0$, $t\in [0,T)$ with $t_0\le t$. 
On the other hand, when $g_{jk,a}(t,x)\not \equiv 0$, 
it is also well-known that the energy inequality of this kind fails to hold 
and we are faced with a difficulty of derivative loss in general. 
Therefore we need some restrictions on $g_{jk,a}$ in order to 
control the $L_x^2$-norm of $v(t,x)$ in terms of $G(t,x)$ and the initial data
(see e.g., Chapter 7 of 
\cite{Miz} for more information on this subject). 
Now, let $\mu_k \in \R \setminus\{0\}$ be given and we set 
\[
\Omega_{t_1,t_2}:=
\sup_{t\in[t_1,t_2)}\sum_{|\beta|\leq2}\sum_{j\in I_N}\sum_{k\in \sh{I}_N}
\sum_{a=1}^2 \jb{t}^{-|\beta|+1}\|\op{J}_{{\mu}_{k}}(t)^\beta
g_{jk,a}(t,\cdot)\|_{W^{2-|\beta|,\infty}}
\]
for $0\le t_1<t_2\le T$. We will show that a kind of energy inequality holds 
if $\Omega$ is suitably small. 
More precisely, we have the following:
\begin{prp}\label{lem:smooth3}
Let $t_0\in [0,T)$ be given and put $\Omega=\Omega_{t_0,T}$. 
Suppose that $v$ solves \eqref{smoothv}. 
There exists positive constants $\omega_0$ and $C_0$, not depending on 
$t_0$ and $T$, such that we have
\[
\|v(t,\cdot)\|_{L^2}
\le 
C_0\|v(t_0,\cdot)\|_{L^2}
+
C_0\int_{t_0}^t \left(
 \frac{\Omega}{\jb{\tau}}\|v(\tau,\cdot)\|_{L^2} + \|G(\tau,\cdot)\|_{L^2}
\right)d\tau 
\] 
for $t\in [t_0,T)$, provided that $\Omega \le \omega_0$.
\end{prp}
We are going to give an outline of the proof. 
Since the idea is essentially not new, we shall be brief. 
For the technical details, see Section 5 of \cite{IKS} 
and the references cited therein. 
Our strategy is to choose an $L^2$-automorphism $\op{S}$ 
(depending on $t\in \R$ and  a parameter $\kappa \in (0,1]$) and 
weight functions $w_a(t,x)$ appropriately so that 
\begin{align}
 [\op{L}_{m_j}, \op{S}] \simeq 
\frac{-i\kappa}{|m_j|\jb{t}} \sum_{a=1}^{2} w_a^2 \op{S}|\pa_{x_a}|
 +\ \mbox{`harmless terms'}, 
\label{LScomm}
\end{align}
where $|\pa_{x_a}|=\op{F}^{-1}|\xi_a|\op{F}$, and to cancel the worst 
contribution from $g_{jk,a}\pa_{x_a}$ in \eqref{smoothv} 
by the first term of the right-hand side of \eqref{LScomm} 
with a suitable choice of $\kappa$.
This plan is carried out as follows: 
let $\op{H}_a$ be the Hilbert transform with respect 
to $x_a$ ($a=1$, $2$), that is,
\[
\bigl(\op{H}_a \phi\bigr)(x)
=\frac{1}{\pi}{\rm p.v.}\int_{\R}\phi(x-\lambda{\bf 1}_a)
\frac{d\lambda}{\lambda},
\]
where ${\bf 1}_a=(\delta_{ab})_{b=1,\,2}\in\R^2$. As in \cite{IKS}, 
we put $\Theta_{a}(t,\,x)=\arctan(x_a/\jb{t})$ and
\[
\op{S}_{\pm,\,a}(t;\kappa)\phi
=(\cosh \kappa \Theta_{a}(t,\cdot))\phi 
\, \mp \, 
i(\sinh \kappa \Theta_{a}(t,\cdot))\op{H}_a \phi
\]
for $t\in \R$, $\kappa \in (0,1]$, $a=1,2$. 
We define 
$\op{S}_\pm(t;\kappa):=\op{S}_{\pm, 1}(t;\kappa) \op{S}_{\pm,\,2}(t;\kappa)$.
Then we can check that 
both $\op{S}_\pm$ and its inverse  $\op{S}_\pm^{-1}$
are bounded operators on $L^2(\R^2)$ with the estimates
\[
\sup_{t\in\R,\ \kappa\in(0,\,1]}\|\op{S}_\pm(t;\kappa)\|_{L^2\rightarrow L^2} 
< \infty ,\ 
\sup_{t\in\R,\ \kappa\in(0,\,1]}\|\op{S}_\pm^{-1}(t;\kappa)\|_{L^2\rightarrow L^2}
<\infty.
\]
As a consequence we have
\begin{align}\label{L2bdd}
 C_*^{-1} \|\phi\|_{L^2}
 \le 
 \|\op{S}(t;\kappa)\phi\|_{L^2}
 \le 
 C_* \|\phi\|_{L^2}
\end{align}
with some $C_*\geq1$ not depending on $t$ and $\kappa$. 
We also set 
\[
w_a(t,x):= \left(1+\frac{x_a^2}{1+t^2}\right)^{-1/2}
=\jbf{\dfrac{x_a}{\jb{t}}}^{-1}.
\]
Note that 
$\pa_{x_a} \Theta_{b} = \delta_{a,b} {\jb{t}}^{-1} w_b^2$. 
With these notations, we have the following key lemmas 
whose proof can be found in  Appendix of \cite{IKS}. 
\begin{lem}\label{smooth1}
Let $m \in \R\setminus\{0\}$ and $\kappa \in (0, 1]$. 
Put $\op{S}(t)=\op{S}_{+}(t; \kappa)$ when $m>0$ and 
$\op{S}(t)=\op{S}_{-}(t; \kappa)$ when $m<0$. 
We have 
\begin{align*}
&\dfrac{d}{dt}\|\op{S}(t)\phi(t,\cdot)\|_{L^2}^2 + 
\dfrac{\kappa}{|m|\jb{t}}
\sum_{a=1}^2
\Bigl\| w_a(t,\cdot)\op{S}(t)|\pa_{x_a}|^{\frac{1}{2}}\phi(t,\cdot) \Bigr\|_{L^2}^2 
\\
&\quad
\leq 
\dfrac{C_1\kappa}{\jb{t}}\|\op{S}(t)\phi(t,\cdot)\|^2_{L^2}
+ 2\Bigl|\Jb{\op{S}(t)\phi(t,\cdot), \op{S}(t)\op{L}_{m}\phi(t,\cdot)
}_{L^2}\Bigr| 
\end{align*}
for $t\geq0$, 
where the constant $C_1$ is independent of $\kappa \in (0,1]$.
\end{lem}
\begin{lem}\label{smooth2}
Let $\kappa \in (0, 1]$ and let 
$\op{S}(t)$, $\op{S}'(t)$ be either $\op{S}_+(t;\kappa)$ or 
$\op{S}_-(t;\kappa)$. We have 
\begin{align*}
&\Bigl|
  \Jb{
  \op{S}(t)\phi, \op{S}(t)\bigl( g(t,\cdot)\pa_{x_a}\psi \bigr)
 }_{L^2}
 \Bigr|
 +
 \Bigl|
   \Jb{
 \op{S}(t)\phi, \op{S}(t)\bigl( g(t,\cdot)\cc{\pa_{x_a}\psi} \bigr)
 }_{L^2} 
 \Bigr| \\
&\quad \leq
C_2 
\bigl(
 \|w_a(t,\cdot)^{-2}g(t,\cdot)\|_{L^\infty} 
 +  
 \|w_a(t,\cdot)^{-1}\pa_{x_a}g(t,\cdot)\|_{L^\infty}
\bigr)
 \\
&\qquad 
\times 
\left(
 \|\phi\|_{L^2} +
 \bigl\|
  w_a(t,\cdot)\op{S}(t)|\pa_{x_a}|^{\frac{1}{2}}\phi 
 \bigr\|_{L^2}
 \right)
\left(
 \|\psi\|_{L^2}  + 
 \bigl\|
   w_a(t,\cdot)\op{S}^\prime(t)|\pa_{x_a}|^{\frac{1}{2}}\psi 
 \bigr\|_{L^2}
\right)
\end{align*}
for $t\in \R$, 
where the constant $C_2$ is independent of $\kappa \in (0,1]$.
\end{lem}
Now we are ready to prove Proposition~\ref{lem:smooth3}. 
Let $\kappa\in(0,1]$ be a parameter to be fixed. 
For each $k\in I_N$ we put 
$\op{S}_k(t)=\op{S}_+(t;\kappa)$ if $m_k>0$, and 
$\op{S}_k(t)=\op{S}_-(t;\kappa)$ if $m_k<0$.
By the relation 
\[
\dfrac{x_a}{\jb{t}}
=
\dfrac{1}{\jb{t}}\op{J}_{\mu_{k},a}-\dfrac{it}{\mu_{k}\jb{t}}\pa_{x_a},
\] 
we have
\begin{align*}
\left\|\jbf{\dfrac{x_a}{\jb{t}}}^2g_{jk,a}(t,\cdot)\right\|_{L^\infty}
+
\left\|\jbf{\dfrac{x_a}{\jb{t}}}\pa_{x_a}g_{jk,a}(t,\cdot)\right\|_{L^\infty} 
&\leq 
 \dfrac{C}{\jb{t}}\sum_{|\beta|\leq 2}\jb{t}^{-|\beta|+1}
\|\op{J}_{\mu_{k}}(t)^{\beta} g_{jk,a}(t,\cdot)\|_{W^{2-|\beta|,\infty}}\\
&\leq 
\dfrac{C\Omega}{|m_j|\jb{t}}.
\end{align*} 
Therefore, by Lemma \ref{smooth2}, we get
\begin{align*}
\sum_{j\in I_N}
 &\Bigl|\Jb{
  \op{S}_j(t) v_j(t,\cdot), \op{S}_j(t)\op{L}_{m_j}v_j(t,\cdot)
 }_{L^2} \Bigr| \\
\le & 
C^\ast 
\sum_{j\in I_N}\sum_{a=1}^2 
\dfrac{\Omega}{|m_j|\jb{t}}\left(\|v_j(t,\cdot)\|_{L^2}^2+
\|w_a(t,\cdot)\op{S}_j(t)|\pa_{x_a}|^{\frac{1}{2}}v_j(t,\cdot)\|_{L^2}^2\right)\\
&+
\sum_{j\in I_N}
\|\op{S}_j(t)v_j(t,\cdot)\|_{L^2} \|\op{S}_j(t)G_j(t,\cdot)\|_{L^2}
\end{align*}
with a positive constant $C^\ast$ independent of $\kappa$. 
We put  $\omega_0=1/(2C^\ast)$ and $\kappa=2C^\ast \Omega$. 
Then it follows from Lemma \ref{smooth1} and \eqref{L2bdd} that
\begin{align*}
\dfrac{d}{dt}\sum_{j\in I_N}\|\op{S}_{j}(t)v_j(t,\cdot)\|^2_{L^2}
&\leq
-\sum_{j\in I_N} \sum_{a=1}^2 \dfrac{\kappa}{|m_j|\jb{t}}
 \Bigl\|
  w_a(t,\cdot)\op{S}_j(t)|\pa_{x_a}|^{\frac{1}{2}}v_j(t,\cdot) 
 \Bigr\|_{L^2}^2\\
&\quad
+\sum_{j\in I_N}
\left(
 \dfrac{C\kappa}{\jb{t}} \|\op{S}_j(t)v_j(t,\cdot)\|_{L^2}^2
 +
 2\Bigl|
   \Jb{\op{S}_j(t) v_j(t,\cdot), \op{S}_j(t)\op{L}_{m_j}v_j(t,\cdot)}_{L^2}
  \Bigr|
\right)\\
&\leq 
 \sum_{j\in I_N}\sum_{a=1}^2
 \dfrac{2C^\ast \Omega -\kappa}{|m_j|\jb{t}}
 \Bigl\| w_a(t,\cdot)S_j(t)|\pa_{x_a}|^{\frac{1}{2}}v_j(t,\cdot) \Bigr\|_{L^2}^2\\
&\quad
 +C\sum_{j\in I_N}
  \left(
   \dfrac{\kappa}{\jb{t}} \|\op{S}_j(t)v_j(t,\cdot)\|_{L^2}^2
   +
   \|S_j(t)v_j(t,\cdot)\|_{L^2}\|\op{S}_j(t)G_j(t,\cdot)\|_{L^2}
  \right)\\
&\leq 
 C
 \left(
   \dfrac{\Omega}{\jb{t}}\|v(t,\cdot)\|_{L^2} + \|G(t,\cdot)\|_{L^2}
 \right)
 \left(\sum_{j\in I_N}\|\op{S}_{j}(t)v_j(t,\cdot)\|^2_{L^2}\right)^{1/2}.
\end{align*}
Integrating with respect to $t$ and using \eqref{L2bdd} again, 
we arrive at the desired result.
\qed

\section{A priori estimate and bootstrap argument}  \label{sec_apriori}
In this section we introduce an a priori estimate for the solution $u$ to 
\eqref{NLS} which leads to Theorem~\ref{thm_sdge} by means of 
the so-called bootstrap argument.

Let $T \in (0,+\infty)$ and let 
$u=(u_j)_{j \in I_N}\in C([0, T];\bigcap_{k=0}^5 H^{11-k, k})$ 
be a solution to \eqref{NLS} for $t \in [0, T)$. 
We set $A_{j} (t,\xi):=\op{F}_{m_j} [\op{U}^{-1}_{m_j}(t)u_j(t,\cdot)](\xi)$, 
$A(t,\xi)=(A_j(t,\xi))_{j\in I_N}$, and define
\begin{align*}
E(T):=
\sup_{0 \leq t <T}\Bigl[ 
(1+t)^{-\delta} \sum_{|\beta| \leq 5} \sum_{j\in I_N} 
\| \op{J}_{m_j}(t)^{\beta}u_j(t, \cdot) \|_{H^{11-|\beta|}} 
+ \sup_{\xi \in \R^2} | \jb{\xi}^{8} A(t, \xi) | 
\Bigr]
\end{align*}
with $\delta \in (0,1/4)$. Then we have the following:
\begin{prp}\label{A priori}
Let $u$, $A$ and $E$ be as above. 
Assume the conditions {\rm ({\bf a})} and {\rm ({\bf b})} are satisfied. 
There exist positive constants $C_3$, $C_4$, $C_5$, $C_6$ and $\eps_1$, 
not depending on $T$, such that
the estimate $E(T) \leq \eps^{2/3}$ 
implies
\begin{align}
| \pa_{\xi}^{\alpha} A(t, \xi) | 
\leq 
\frac{C_3 \eps \,(1+t)^{C_4|\alpha|\eps^{1/3}}}{\jb{\xi}^{8-|\alpha|}} 
\label{kakutenAA}
\end{align}
for $(t,\xi) \in [0, T)\times \R^2$, $|\alpha|\leq 3$ and 
\begin{align}
\sum_{j\in I_N}\| \op{J}_{m_j}(t)^{\beta}u_j(t, \cdot) \|_{H^{11-|\beta|}} 
\le C_5 \eps (1+t)^{C_6 \eps^{1/3}}
\label{L2_est}
\end{align}
for $t\in [0,T)$, $|\beta|\le 5$, 
provided that $\eps=\sum_{k=0}^5 \|\varphi \|_{H^{11-k,k}} \leq \eps_1$.
\end{prp}

\begin{rmk} 
The indices `$11$' and `$5$' appear by technical reasons. 
One may improve this point, but we do not address it here.
\end{rmk}

\begin{cor}\label{cor_apriori}
Under the same assumptions as above, there exist positive constants $K$ and 
$\eps_2$, not depending on $T$, such that
the estimate $E(T) \leq \eps^{2/3}$ implies the better estimate 
$E(T) \leq K \eps$ if $\eps\le \eps_2$.
\end{cor}

This proposition will be proved in Section 5. 
In the rest part of this section, we will derive Theorem~\ref{thm_sdge} 
from Corollary~\ref{cor_apriori}. 
First let us recall the local existence theorem. For fixed $t_0\geq 0$, 
let us consider the Cauchy problem
\begin{align}\label{NLS_loc}
\left\{
\begin{array}{l}
\op{L}_{m_{j}} u_{j} =F_j(u, \pa_x u), \quad t>t_0,\, x\in\R^2, \, j\in I_N,\\
u_j(t_0,x) = \psi_j(x), \qquad x\in\R^2,\, j\in I_N. 
\end{array}
\right.
\end{align}

\begin{lem}\label{local}
Let $\psi=(\psi_j)_{j\in I_N}\in \bigcap_{k=0}^5 H^{11-k, k}$. 
There exists a positive constant $\eps_0$, which is independent of $t_0$, 
such that the following holds: 
for any $\underline{\eps}\in (0,\,\eps_0]$ and $M\in (0,\, \infty)$, 
one can choose a positive constant 
$\tau^\ast=\tau^\ast(\underline{\eps},\,M)$, 
which is independent of $t_0$, 
such that \eqref{NLS_loc} admits a unique solution 
$u=(u_j)_{j\in I_N}\in C([t_0,\,t_0+\tau^\ast];\bigcap_{k=0}^5 H^{11-k, k})$, 
provided that
\[
\|\psi\|_{H^5}
\leq \underline{\eps} ~\ and\ ~ 
\sum_{|\beta|\le 5}   \sum_{j \in I_N} 
\|\op{J}_{m_j}(t_0)^\beta \psi_j \|_{H^{11-|\beta|}} \leq M.
\]
\end{lem}
We skip the proof of this lemma because it is standard
(see e.g., Appendix of \cite{IKS} for the proof of similar lemma).

Now we are going to prove Theorem~\ref{thm_sdge}. The argument below is 
almost parallel to that of \S 6.1 in \cite{LS1}. 
Let $T^\ast$ be the supremum of all $T\in(0,\,\infty)$ such that the problem  
\eqref{NLS} admits a unique solution 
$u=(u_j)_{j\in I_N}\in C([0,\,T];\bigcap_{k=0}^5 H^{11-k, k})$. 
By Lemma \ref{local} with $t_0=0$, we have $T^\ast>0$ if 
$\|\varphi\|_{H^5}\leq \eps <\eps_0$.  
We also set 
\[
T_\ast = \sup \{ \tau\in[0,\,T^\ast) | E(\tau)\leq \eps^{2/3}\}.
\] 
By Sobolev imbedding and Corollary~\ref{cor_apriori}, 
we have 
\[
E(0)\leq \eps +C\|\varphi\|_{H^{8,2}}\leq C\eps \leq \dfrac{1}{2}\eps^{2/3}
\] 
if $\eps$ is small enough. 
Note that $T_\ast>0$ because of the continuity of 
$[0,T^\ast)\ni \tau \mapsto E(\tau)$. 
We claim that $T^\ast=T_\ast$ if $\eps$ is sufficiently small. 
Indeed, if $T_\ast < T^\ast$, Corollary \ref{A priori} with $T=T_\ast$ yields 
$E(T_\ast)\leq K\eps\leq\frac{1}{2}\eps^{2/3}$ for 
$\eps\leq\eps_3=\min\{\eps_2,\,1/(2K)^3\}$ 
where $K$ and $\eps_2$ are mentioned in Corollary~\ref{cor_apriori}. 
By the continuity of $[0,\,T^\ast)\ni\tau\mapsto E(\tau)$, 
we can take $T^\flat\in(T_\ast,\,T^\ast)$ such that 
$E(T^\flat)\leq \eps^{2/3}$, which contradicts the definition of $T_\ast$. 
Therefore we must have $T_\ast=T^\ast$. 
By using Corollary~\ref{cor_apriori} with $T=T^\ast$ again, we see that
\[
\sum_{j \in I_N}\sum_{|\beta|\leq5}
\|\op{J}_{m_j}(t)^\beta u_j(t,\cdot)\|_{H^{11-|\beta|}}
\leq 
K\eps(1+t)^\delta,
\qquad 
\sum_{j \in I_N}\sup_{\xi\in\R^2}|\jb{\xi}^8 A_j(t,\,\xi)|
\leq K\eps
\]
for $t\in[0,\,T^\ast)$. In particular, by Lemma~\ref{lem3} , we have
\[
\sup_{t\in[0,\,T^\ast)}\|u(t,\,\cdot)\|_{H^5}
\leq 
C\sum_{1 \leq j \leq N}
\sup_{(t,\,\xi)\in[0,\,T^\ast)\times\R^2}|\jb{\xi}^{5+2}A_j(t,\,\xi)|
\leq 
C^\flat\eps
\]
with some $C^\flat>0$. 
Next we assume $T^\ast <\infty$. 
Then, by setting $\eps_4=\min\{\eps_3,\,\eps_0/2C^\flat\}$, 
$M=K\eps_4(1+T^\ast)^\delta$, we have
\[
\sup_{t\in[0,\,T^\ast)}\sum_{j \in I_N}
\sum_{|\beta|\leq5}\|\op{J}_{m_j}(t)^\beta u_j(t,\cdot)\|_{H^{11-|\beta|}}
%
%
\leq M
\]
as well as
\[
\sup_{t\in[0,\,T^\ast)}\|u(t,\,\cdot)\|_{H^5}
\leq 
C^\flat\eps \leq \eps_0/2 <\eps_0
\]
for $\eps \leq \eps_4$. 
By Lemma \ref{local}, there exists $\tau^\ast>0$ such that 
\eqref{NLS} admits the solution 
$u=(u_j)_{j\in I_N}\in C([0,\,T^\ast+\tau^\ast];\bigcap_{k=0}^5 H^{11-k, k})$ if $\eps\leq\eps_0$. 
This contradicts the definition of $T^\ast$, 
which means $T^\ast=\infty$ for $\eps\leq\min\{\eps_0,\eps_4\}$. 
Moreover, we have 
\[
\Bigl(\ \|u(t,\cdot)\|_{L^2 \cap L^{\infty}}\lesssim \ \Bigr)\ 
\|u(t,\cdot)\|_{H^2}
\leq 
C\sum_{j \in I_N}\sup_{(\tau,\,\xi)\in[0,\,\infty)\times\R^2}
|\jb{\xi}^{4}A_j(\tau,\,\xi)|
\leq C\eps
\]
and
\[
\|u(t,\cdot)\|_{L^{\infty}}
\leq 
\dfrac{C}{t}\| A(t,\cdot)\|_{L^\infty}
+\dfrac{C}{t^{3/2}}
\sum_{j\in I_N}\sum_{|\beta|\leq 2}
 \|\op{J}_{m_j}(t)^\beta u_j(t,\cdot)\|_{L^2}
\leq
\dfrac{C\eps}{t},\qquad t\ge 1,
\]
by Lemmas~\ref{lem3} and \ref{Linftylem}, respectively. 
This completes the proof of Theorem \ref{thm_sdge}.

\qed

\section{Proof of Proposition~\ref{A priori}}  
\label{sec_pf_a_priori}

This section is devoted to the proof of Proposition~\ref{A priori}.
Throughout this section, 
we always assume that the conditions ({\bf a}), ({\bf b}) are satisfied, 
and that $u\in C([0,T];\bigcap_{k=0}^5 H^{11-k,k}(\R^2))$ 
is a solution to \eqref{NLS} which satisfies
\begin{align} \label{assump_apriori}
 E(T) \le \eps^{2/3}
\end{align}
for given $T>0$. 
The proof will be divided into three parts: 
we first consider the case 
of $t\in [0,1]$ in \S~\ref{bd_small_time}, and then  we will show 
\eqref{kakutenAA} and \eqref{L2_est} in \S~\ref{sec:A} and \S~\ref{L^2}, 
respectively.
In what follows, we will use the following convention on implicit 
constants: 
the expression $f=\sum_{\lambda \in \Lambda}' g_{\lambda}$ means that there 
exists a family $\{C_{\lambda}\}_{\lambda \in \Lambda}$ of constants such 
that $f=\sum_{\lambda \in \Lambda} C_{\lambda} g_{\lambda}$.

\subsection{Estimates in the small time}\label{bd_small_time}
In this part, we focus on the case of $t\in [0,1]$. This case is easier 
because we do not have to pay attentions to possible growth in $t$. 

Let $\gamma\in \Z_+^2$ satisfy $|\gamma|\le 5$. 
By the Sobolev imbedding $H^2(\R^2)\hookrightarrow L^\infty(\R^2)$ 
and the assumption \eqref{assump_apriori}, we have
\begin{align*}
\frac{d}{dt} \|\op{J}_{m_j}^\gamma u_j\|_{H^{8-|\gamma|}}
&\leq 
\|\op{J}_{m_j}^\gamma F_j(u, \pa_x u)\|_{H^{8-|\gamma|}} \\
&\leq 
C\sum_{|\alpha|,|\beta|\leq 1}\sum_{k,l \in \sh{I}_N}
\sum_{\gamma' + \gamma'' = \gamma}
\bigl\|
 (\op{J}^{\gamma'}_{\tilde{m}_{k}}\pa_x^{\alpha}\sh{u}_{k}) 
(\op{J}_{\tilde{m}_{l}}^{\gamma''}\pa_x^{\beta} \sh{u}_{l} )
\bigr\|_{H^{8-|\gamma|}}
 \\
&\leq 
C\sum_{k,l\in I_N}\sum_{|\gamma'| + |\gamma''|\le 5}
 \|\op{J}^{\gamma'}_{m_{k}}u_{k}\|_{H^{9-|\gamma|}} 
\|\op{J}_{m_{l}}^{\gamma''}u_{l}\|_{H^3}\\
&\leq 
C \eps^{4/3} (1+t)^{2\delta},
\end{align*} 
whence
\begin{equation}\label{0_to_1}
\sup_{t\in [0,1]}\sum_{j \in I_N}\sum_{|\gamma| \leq 5} 
\| \op{J}_{m_j}(t)^{\gamma}u_j(t, \cdot) \|_{H^{8-|\gamma|}}
\le
C\eps +C\eps^{4/3} \int_0^{1} (1+\tau)^{2\delta}\, d\tau
\le 
C\eps.
\end{equation}
Therefore Lemma~\ref{lem3} gives us 
\begin{align*}
\sum_{|\alpha|\leq3}
\jb{\xi}^{8-|\alpha|} |\pa_{\xi}^{\alpha}A(t,\xi)|
&\le
C\sum_{|\alpha|\leq 3}
\|A_j(t,\cdot)\|_{H_\xi^{|\alpha|+2, 8-|\alpha|}}\\
&\leq 
C\sum_{j \in I_N}\sum_{|\gamma|\leq 5}
\|\op{J}_{m_j}^{\gamma} u_j(t,\cdot)\|_{H_x^{8-|\gamma|}}\\
&\leq 
 C\eps
\end{align*}
for $(t,\xi)\in [0,1]\times \R^2$.
Next we put 
$v_j^{(\alpha,\beta)}(t,x):=\pa_x^\alpha \op{J}_{m_j}(t)^\beta u_j(t,x)$ 
for $\alpha,\beta\in \Z_+^2$ with $|\alpha|+|\beta|\le 11$, $|\beta|\le 5$. 
We also set 
\begin{align}
G_j^{(\alpha,\beta)}:=&
\pa_x^\alpha \op{J}_{m_j}^\beta  
F_j(u,\pa_x u)
\nonumber\\
&-
\sum_{\substack{ |\alpha'|= 1\\ |\beta'|\le 1}}
\sum_{k,l\in \sh{I}_N}
C_{j,k,l}^{\alpha',\beta'}
 \Bigl(\frac{\tilde{m_k}}{m_j}\Bigr)^{|\beta|}
 \sh{(\pa_x^{\alpha'}\pa_x^\alpha \op{J}_{m_k}^\beta u_k)}
 (\pa_x^{\beta'} \sh{u}_{l} ) 
\nonumber\\
 &-
\sum_{\substack{|\alpha'|\le 1\\ |\beta'| = 1}}
\sum_{k,l\in \sh{I}_N}
C_{j,k,l}^{\alpha',\beta'}
 \Bigl(\frac{\tilde{m_l}}{m_j}\Bigr)^{|\beta|}
 (\pa_x^{\alpha'}\sh{u}_{k}) 
 \sh{(\pa_x^{\beta'} \pa_x^\alpha \op{J}_{m_l}^\beta u_l)}, 
\label{def_G}
\end{align}
where $C_{j,k,l}^{\alpha',\beta'}$ comes from \eqref{nonlin}. Then we have 
\begin{align}
\op{L}_{m_j} v_{j}^{(\alpha,\beta)}
=&\pa_x^\alpha \op{J}_{m_j}^\beta  F_j(u,\pa_x u)\notag\\
=& 
 \sum_{k,l\in \sh{I}_N}\left(
 \sumc_{\substack{ |\alpha'|= 1\\ |\beta'|\le 1}}
  (\pa_x^{\beta'} \sh{u}_{l} ) \sh{\pa_x^{\alpha'}(v_{k}^{(\alpha,\beta)})}
  + 
 \sumc_{\substack{|\alpha'|\le 1\\ |\beta'| = 1}}
 (\pa_x^{\alpha'}\sh{u}_{k}) 
 \sh{\pa_x^{\beta'} (v_{l}^{(\alpha,\beta)})} 
 \right)
+ G_j^{(\alpha,\beta)}\notag\\
=&
\sum_{k\in \sh{I}_N} \sum_{a=1}^2 g_{jk,a}^{(\alpha,\beta)} 
\pa_{x_a}\sh{(v_{k}^{(\alpha,\beta)})} +G_j^{(\alpha,\beta)}, 
\label{dJF}
\end{align}
where $g_{jk,a}^{(\alpha,\beta)}$ is a linear combination of 
$\pa_x^{\gamma} \sh{u}_{l}$ ($|\gamma|\le 1$, $l\in \sh{I}_N$). 
In view of Lemma \ref{leibniz} and the commutation relation \eqref{commute}, 
we see that $G_j^{(\alpha,\beta)}$ can be written in the form
\begin{align}
 G_j^{(\alpha,\beta)}=\sum_{k,l\in \sh{I}_N} 
 \sumc_{
  \substack{
                  |\sigma|+|\sigma'| \le |\alpha|+2\\ 
                   |\rho|+|\rho'| \le |\beta|\\ 
                     \max\bigl\{ |\rho|+|\sigma|,\,  |\rho'|+|\sigma'|  \bigr\}\leq|\alpha|+|\beta|\\ 
                        \max\bigl\{ |\sigma|,\,  |\sigma'|  \bigr\}\leq|\alpha|+\min\{ 1,\,  |\beta| \} 
               }
           }  
\sh{\bigl(\pa_x^{\sigma} \op{J}_{m_k}^{\rho} u_k \bigr)}
\sh{\bigl(\pa_x^{\sigma'} \op{J}_{m_{l}}^{\rho'} u_l \bigr)}.
\label{G_key}
\end{align}
In order to estimate this term, we set
\begin{equation*}
\ene{p,q}(t)
:=
 \sum_{j \in I_N}
 \sum_{|\gamma|\le q}
 \left\|\op{J}_{m_j}(t)^\gamma u_j(t,\cdot)\right\|_{H^{p}} 
\end{equation*}
for $p,\ q\in\Z_+$. 
We also set $\ene{p,q}(t):=0$ for $q\leq -1$. 
Let $|\alpha|=p$, $|\beta|=q$. 
Then, if $p+q\le 11$ and $q\le 5$, 
$G_j^{(\alpha,\beta)}$ 
can be estimated as follows:
\begin{align}
\sum_{j\in I_N}\|G_j^{(\alpha,\beta)}(t,\cdot)\|_{L^2}
\leq 
C
\sum_{|\gamma|\le 3}\sum_{j\in I_N}
\|\op{J}_{m_j}(t)^\gamma u_j(t,\cdot)\|_{W^{8-|\gamma|,\infty}}
 \ene{p+\min\{ 1,\,  |\gamma| \},q-|\gamma|}(t)  \label{est_yayakoshii}
\end{align}
(see \S~\ref{sec:A2} for the derivation of this inequality).
Therefore we have
\begin{align*}
 \|G_j^{(\alpha,\beta)}(t,\cdot)\|_{L^2} 
\le& 
C \eps^{5/3}(1+t)^{\delta}
\end{align*}
under the condition ({\bf a}) and the assumption \eqref{assump_apriori}.
We also note that 
\begin{align*}
\sum_{|\beta'|\leq2}\sum_{l\in \sh{I}_N}\sum_{|\gamma|\leq 1}
\jb{t}^{-|\beta'|+1}
\|
  \op{J}_{\tilde{m}_l}^{\beta'} \pa_x^{\gamma} \sh{u}_{l} 
\|_{W^{2-|\beta'|,\infty}}
\leq 
C\sum_{j\in I_N}\sum_{|\beta'|\leq 2}
\|\op{J}_{m_j}^{\beta'} u_j \|_{H^{5}}
\leq 
 C\eps 
\end{align*}for $t\in[0,1]$. 
Therefore 
we can apply Proposition~\ref{lem:smooth3} to \eqref{dJF} and conclude that
\begin{align*}
\sum_{|\beta|\le 5}
\|\op{J}_{m_j}^\beta u_j(t,\cdot)\|_{H^{11-|\beta|}}
&\le
 C\sum_{\substack{ |\alpha|+|\beta|\le 11\\ |\beta|\le 5}}
\|v_j^{(\alpha,\beta)}(t)\|_{L^2} \\
&\leq 
C\eps + 
C\int_0^1 
 \left(
  \dfrac{C\eps}{\jb{\tau}} \eps^{2/3}(1+\tau)^\delta
  + 
  C\eps^{5/3} (1+\tau)^{\delta}
 \right)d\tau\\
&\leq  
C\eps
\end{align*}
for $t\in[0,1]$, as desired.

\subsection{Pointwise estimates in the large time}\label{sec:A}
The goal of this part is to obtain 
\begin{align}\label{est_A}
| \pa_{\xi}^{\gamma} A(t, \xi) | 
\leq 
\frac{C\eps \, t^{C|\gamma|\eps^{1/3}}}{\jb{\xi}^{8-|\gamma|}}
\end{align}
for $(t,\xi) \in [1, T)\times \R^2,\ |\gamma|\leq 3$ 
under the assumption \eqref{assump_apriori}. 
To this end, we set 
\begin{equation}\label{R}
R_j(t,\xi):=
\op{F}_{m_j} [\op{U}^{-1}_{m_j}F_j(u(t,\cdot),\pa u(t,\cdot))](\xi)-\frac{1}{t}p_j(\xi,A(t,\xi))
\end{equation}
so that
\begin{align}\label{red1}
i\pa_tA_j(t,\xi)
&=
 \op{F}_{m_j} \op{U}^{-1}_{m_j}(t)\op{L}_{m_j}u_j \notag \\
&=
 \op{F}_{m_j} \op{U}^{-1}_{m_j}(t)F_j(u,\pa u) \notag \\
&=
 \frac{1}{t}p_j(\xi;A(t,\xi)) + R_j(t,\xi)
\end{align}
for each $j\in I_N$. 
In view of the folloing lemma, we see that $R(t,\xi)=(R_j(t,\xi))_{j\in I_N}$ 
can be regarded as a remainder if we have a good control of 
$\|\op{J}_{m_j}^{\beta}u_j\|_{H^{11-|\beta|}}$ for $|\beta|\le 5$.
\begin{lem}\label{R4}
Suppose that the condition {\rm ({\bf a})} is satisfied. 
For $k\in \Z_+$ and $\gamma \in \Z_+^2$, we have
\[
| \pa_{\xi}^{\gamma}R(t,\xi)| 
\leq 
\frac{C}{t^{3/2}\jb{\xi}^k}
\sum_{j \in I_N}\sum_{|\beta|\leq |\gamma|+2}
\|\op{J}_{m_j}(t)^{\beta}u_j(t,\cdot)\|_{H^{k+1}}^2, 
\quad (t,\xi)\in [1,T)\times \R^2.
\]
\end{lem}
We will give the proof of this lemma in \S~\ref{sec:R4}. 
(The proof looks a bit complicated, but the idea is quite simple: 
split $\pa_{\xi}^{\gamma} R$ into a linear combination of terms 
including the factor $t^{-1}(\op{W}^{\pm}-1)$, 
and apply Lemma~\ref{W} to each of them.) 
%
%
Anyway, what we need here is 
\begin{align}\label{est_R}
|\pa_{\xi}^{\gamma} R(t,\xi)| 
\leq 
\frac{C\eps^{4/3}}{t^{3/2 -2\delta} \jb{\xi}^{8-|\gamma|}},
\quad  (t,\xi)\in [1,T)\times \R^2,
\end{align}
for $|\gamma| \leq 3$, which is a consequence of Lemma~\ref{R4} and 
the assumption \eqref{assump_apriori}.
Note that $(8-|\gamma|)+1=11-(|\gamma|+2)$.

Now we are going to prove \eqref{est_A}. First we consider the case of 
$\gamma=0$.
We put 
\[
 \nu(t,\xi) = \sqrt{\jb{A(t,\xi),HA(t,\xi)}_{\C^N}},
\] 
where $H$ is the positive Hermitian matrix appearing in the condition 
{\rm ({\bf b})}. Remark that
\begin{equation}\label{equ}
\sqrt{\eta_\ast}|A(t,\xi)|
\leq
\nu(t,\xi)\leq\sqrt{\eta^\ast}|A(t,\xi)|
\end{equation}
where $\eta_\ast$ and $\eta^\ast$ are the smallest and largest eigenvalues 
of $H$, respectively. It follows from \eqref{red1} and {\rm ({\bf b})} that
\begin{align*}
\pa_t \nu(t,\xi)^2 &= 2\imagpart \jb{i\pa_t A(t,\xi),HA(t,\xi)}_{\C^N}\\
&=
\frac{2}{t}\imagpart \jb{p(\xi;A(t,\xi)),HA(t,\xi)}_{\C^N}
+2\imagpart \jb{R(t,\xi),HA(t,\xi)}_{\C^N}\\
&\leq 
0+C|R(t,\xi)|\nu(t,\xi).
\end{align*}
By \eqref{est_R}, we have
\begin{equation}\label{alpha0}
\nu(t,\xi)\leq \nu(1,\xi)+C\int^t_1|R(\tau,\xi)|d\tau
\leq 
\frac{C\eps}{\jb{\xi}^8}
+
\frac{C\eps^{4/3}}{\jb{\xi}^8}\int^\infty_1\frac{d\tau}{\tau^{3/2-2\delta}}
\leq
\frac{C\eps}{\jb{\xi}^8}.
\end{equation}
Therefore we obtain
\[
|A_j(t,\xi)|\leq C\nu(t,\xi)\leq \frac{C\eps}{\jb{\xi}^8},
\]
as desired. 

Next we turn our attentions to the case of $1\le |\gamma|\le 3$ 
in \eqref{est_A}. Before doing so, we set 
$\Lambda_{j,k,l}^{\alpha,\beta}(\xi)
= 
C_{j,k,l}^{\alpha,\beta}(i \tilde{m}_{k}\xi)^{\alpha}
(i \tilde{m}_{l}\xi)^{\beta}$ 
so that $p_j(\xi;Y)$ can be expressed by
\begin{equation*}
p_j(\xi;Y)
=\sum_{|\alpha|,|\beta|\leq 1}\sum_{k,l \in \sh{I}_N} 
\Lambda_{j,k,l}^{\alpha,\beta}(\xi)\sh{Y}_{k}\sh{Y}_{l}.
\end{equation*}
Applying $\pa_{\xi}^{\gamma}$ to \eqref{red1}, we have 
\[
 i\pa_t \pa_{\xi}^{\gamma}A_j
 =\frac{1}{t} \sum_{|\alpha|,|\beta|\leq 1}\sum_{k,l \in \sh{I}_N} 
 \sumc_{\gamma'+\gamma''+\gamma'''=\gamma}
(\pa_{\xi}^{\gamma'}\Lambda_{j,k,l}^{\alpha,\beta})
(\pa_{\xi}^{\gamma''}\sh{A}_{k})(\pa_{\xi}^{\gamma'''}\sh{A}_{l})
+\pa_{\xi}^{\gamma}R_j.
\]
By virtue of \eqref{est_R}, we see that
\[
 |\pa_t \pa_{\xi}^{\gamma}A(t,\xi)|
 \le \frac{C}{t}
\sum_{\substack{\gamma'+\gamma''+\gamma'''=\gamma \\ |\gamma'|\le 2}}
\jb{\xi}^{2-|\gamma'|}
|\pa_{\xi}^{\gamma''}A(t,\xi)| |\pa_{\xi}^{\gamma'''}A(t,\xi)|
+\frac{C\eps^{4/3}}{t^{3/2 -2\delta} \jb{\xi}^{8-|\gamma|}}.
\]
Now we take $|\gamma|=1$. It follows that 
\begin{align*}
  |\pa_t \pa_{\xi}A(t,\xi)|
 &\le 
 \frac{C}{t}
 \Bigl( \jb{\xi}|A|^2 +\jb{\xi}^2 |A||\pa_{\xi}A| \Bigr)
 + 
 \frac{C\eps^{4/3}}{t^{3/2 -2\delta} \jb{\xi}^{7}}\\
 &\le
 \frac{C}{t}
 \Bigl(\eps^2 \jb{\xi}^{-15} + \eps \jb{\xi}^{-6} |\pa_{\xi}A| \Bigr)
 + 
 \frac{C\eps^{4/3}}{t^{3/2 -2\delta} \jb{\xi}^{7}}\\
 &\le
 \frac{C\eps }{t}|\pa_{\xi}A(t,\xi)|  +\frac{C\eps^{4/3}}{t\jb{\xi}^7}.
\end{align*}
Hence we deduce from the Gronwall-type argument that
\begin{align*}
|\pa_{\xi}A(t,\xi)|
&\leq 
|\pa_{\xi}A(1,\xi)|t^{C\eps}
+\int^{t}_1\frac{C\eps^{4/3}}{\tau \jb{\xi}^7}
\left(\frac{t}{\tau}\right)^{C\eps}d\tau\\
&\leq 
\frac{C\eps t^{C\eps}}{\jb{\xi}^7}
+\frac{C\eps^{4/3}}{\jb{\xi}^7}
\int^{t}_1\frac{1}{\tau}\left(\frac{t}{\tau}\right)^{C\eps^{1/3}}d\tau\\
&\leq 
\frac{C\eps t^{C\eps^{1/3}}}{\jb{\xi}^7},
\end{align*}
as required.
Next we take $|\gamma|=2$. Then we have as before that 
\begin{align*}
  |\pa_t \pa_{\xi}^{\gamma}A(t,\xi)|
 &\le 
 \frac{C}{t}
 \Bigl\{ \jb{\xi}|A||\pa_{\xi}A| 
+\jb{\xi}^2( |\pa_{\xi}A|^2+|A||\pa_{\xi}^{\gamma}A|) \Bigr\}
 + 
 \frac{C\eps^{4/3}}{t^{3/2 -2\delta} \jb{\xi}^{6}}\\
 &\le
 \frac{C}{t}
 \Bigl(\eps^2 \jb{\xi}^{-12}t^{C\eps^{1/3}}
 + \eps \jb{\xi}^{-6} |\pa_{\xi}^\gamma A| \Bigr)
 + 
 \frac{C\eps^{4/3}}{t^{3/2 -2\delta} \jb{\xi}^{6}}\\
 &\le
 \frac{C\eps }{t}|\pa_{\xi}^{\gamma}A(t,\xi)| 
 +\frac{C\eps^{4/3}}{t^{1-C\eps}\jb{\xi}^6}.
\end{align*}
So, the Gronwall-type argument again implies
\begin{align*}
\sum_{|\gamma|= 2}|\pa_{\xi}^{\gamma}A(t,\xi)|
\leq 
\sum_{|\gamma|= 2}|\pa_{\xi}^{\gamma} A(1,\xi)|t^{C\eps}
+\frac{C\eps^{4/3}}{\jb{\xi}^6}
\int^{t}_1
\frac{1}{\tau^{1-C\eps}}\left(\frac{t}{\tau}\right)^{C\eps^{1/3}}d\tau
\leq 
\frac{C\eps t^{C\eps^{1/3}}}{\jb{\xi}^6}.
\end{align*}
Note that $\eps\ll \eps^{1/3}$ for small $\eps$. 
Similary, when $|\gamma|=3$ we have 
\begin{align*}
  |\pa_t \pa_{\xi}^{\gamma}A(t,\xi)|
 &\le
 \frac{C\eps }{t}|\pa_{\xi}^{\gamma}A(t,\xi)| 
 +\frac{C\eps^{4/3}}{t^{1-C\eps}\jb{\xi}^5},
\end{align*}
whence 
\begin{align*}
\sum_{|\gamma|= 3}|\pa_{\xi}^\gamma A(t,\xi)|
\leq 
\frac{C\eps t^{C\eps^{1/3}}}{\jb{\xi}^5}.
\end{align*}
This completes the proof of \eqref{est_A} for all $|\gamma|\le 3$.
\qed

\subsection{$L^2$-estimates in the large time  }\label{L^2}
The remaining task is to show \eqref{L2_est} for $t\in [1,T)$ 
under the assumption \eqref{assump_apriori}. 
Remember that 
\begin{equation*}
\ene{p,q}(t)
=
 \sum_{j \in I_N}
 \sum_{|\beta|\le q}
 \left\|\op{J}_{m_j}(t)^\beta u_j(t,\cdot)\right\|_{H^{p}} 
\end{equation*}
for $p,\ q\in\Z_+$, and we set $\ene{p,q}(t):=0$ for $q\leq -1$.
\begin{lem}\label{L2base}
Let $p$, $q \in \Z_+$ satisfy $q \leq 5$ and $p +q \leq 11$. 
Under the assumption \eqref{assump_apriori}, 
there exist positive constants $C_7$ and $C_8$, not depending on 
$T$ and $\eps$, such that 
\begin{equation}
\ene{p,q}(t)
\leq 
C_7\eps t^{C_8\eps^{1/3}}, 
\quad t\in [1,T).
 \label{L2est}
\end{equation}
\end{lem}
Once this lemma is verified, it is straightforward that we have 
\eqref{L2_est} for $t\in [1,T)$. 
The rest of this subsection is devoted to getting Lemma~\ref{L2base}. 

\begin{proof}
Let $|\alpha|=p$, $|\beta|=q$ and 
$v_j^{(\alpha,\beta)}=\pa_x^\alpha \op{J}_{m_j}^\beta u_j$. 
Remember that $v_j^{(\alpha,\beta)}$ satisfies \eqref{dJF}.
From the argument in \S~\ref{sec:A} and Lemma~\ref{Linftylem}, 
we already know that 
\begin{equation}
\sum_{j\in I_N}\|\op{J}_{m_j}^\gamma(t) u_j(t,\cdot)\|_{W^{8-|\gamma|,\infty}}
\leq  
\frac{C \eps^{2/3}}{t^{1-|\gamma|C\eps^{1/3}}} 
\label{rough}
\end{equation}
for $t\in [1,T)$, $|\gamma|\le 3$. 
By virtue of \eqref{est_yayakoshii} and \eqref{rough}, we have
\[
\sum_{j\in I_N}\|G_j^{(\alpha,\beta)}(t,\cdot)\|_{L^2}
\leq 
C\eps^{2/3}\left(
\frac{\ene{p,q}(t)}{t}
+
\frac{\ene{p+1,q-1}(t)}{t^{1-C\eps^{1/3}}}
\right)
\]
and 
\begin{align*}
\sum_{|\beta'|\leq2}\sum_{l\in \sh{I}_N}\sum_{|\gamma|\leq 1}
\jb{t}^{-|\beta'|+1}
\|
  \op{J}_{\tilde{m}_l}^{\beta'} \pa_x^{\gamma} \sh{u}_{l} 
\|_{W^{2-|\beta'|,\infty}}
\leq 
C\eps^{2/3} \sum_{|\beta'|\leq 2}
t^{-(1-C\eps^{1/3})|\beta'|}
\leq 
C\eps^{2/3}. 
\end{align*}
Therefore 
we can adapt Proposition~\ref{lem:smooth3} to obtain
\begin{align}
\ene{p,q}(t)
\leq 
C\eps
+
C\eps^{2/3}\int_1^t \left(
 \dfrac{\ene{p,q}(\tau)}{\tau}
 + 
\frac{\ene{p+1,q-1}(\tau)}{\tau^{1-C\eps^{1/3}}}\right)\, d\tau.
\label{ene_induc}
\end{align}
Now we shall argue by induction on $q$. 
First we consider the case of $q=0$. By \eqref{ene_induc} with $q=0$, we have 
\[
\ene{p,0}(t)
\leq 
C\eps+ C\eps^{2/3} \int_1^t 
 \dfrac{\ene{p,0}(\tau)}{\tau}\, d\tau.
\]
Hence the Gronwall lemma yields \eqref{L2est} with $q=0$.
Next we assume that \eqref{L2est} is valid for some $0\le q\le 4$. Then 
it follows from the estimate \eqref{ene_induc} with $q$ replaced by $q+1$ 
that 
\[
\ene{p,q+1}(t)\leq 
C\eps+C\eps^{4/3} t^{C\eps^{1/3}}+
C\eps^{2/3} \int_1^t 
 \dfrac{\ene{p,q+1}(\tau)}{\tau}\, d\tau.
\]
Therefore the Gronwall lemma again yields \eqref{L2est} 
with $q$ replaced by $q+1$.
\end{proof}

\appendix \section{Appendix}\label{sec:App}

This section is devoted to the proof of Lemma \ref{R4} and 
the inequality \eqref{est_yayakoshii}.

\subsection{Proof of the Lemma~\ref{R4}}\label{sec:R4}
To make the argument clear, we focus on the case where 
$F_1=(\cc{\pa_{x_1}u_2})(\pa_{x_1}u_3)$ with $m_1+m_2=m_3$. 
General case can be shown in the same way. 
In what follows, we write 
$A_j^{(\alpha)}(t,\xi)=(i m_j \xi)^\alpha A_j(t,\xi)$ 
for $\alpha\in\Z_+^2$. 
Note that we have
\[
\pa_{x}^\alpha u_j
=
\op{U}_{m_j}\op{F}^{-1}_{m_j}A_{j}^{(\alpha)}
=
\op{M}_{m_j}\op{D}\op{W}_{m_j}A^{(\alpha)}_j,
\qquad
\cc{\pa_{x}^\alpha u_j}=\op{M}_{-m_j}\op{D}\op{W}_{-m_j}\cc{A^{(\alpha)}_j}.
\]
We also put $\iota=(1,0)\in \Z_+^2$ so that 
$p_1(\xi;A)=A_2^{(\iota)}\cc{A_3^{(\iota)}}$.

Let us begin with the simplest case $k=|\gamma|=0$. 
By the factrization of $\op{U}_m$ and the relation $m_1=-m_2+m_3$, we have
\begin{align*}
\op{F}_{m_1} \op{U}^{-1}_{m_1} F_1
&=
\op{F}_{m_1} \op{U}^{-1}_{m_1}
\bigl[\cc{(\pa_x^{\iota}u_2})({\pa_x^{\iota}u_3)} \bigr]
\\
&=
\op{W}_{m_1}^{-1} \op{D}^{-1}\op{M}^{-1}_{m_1}
\Bigl[
\bigl( \op{M}_{-m_2}\op{D}\op{W}_{-m_2}\cc{A_2^{(\iota)}} \bigr)
\bigl( \op{M}_{m_3}\op{D}\op{W}_{m_3}{A_3^{(\iota)}} \bigr)
\Bigr]\\
&=
\dfrac{1}{t}\op{W}_{m_1}^{-1} 
\Bigl[
\bigl( \op{W}_{-m_2}\cc{A_2^{(\iota)}} \bigr)
\bigl( \op{W}_{m_3}A_3^{(\iota)} \bigr)
\Bigr].
\end{align*}
Hence $R_1$ can be rewritten as
\begin{align*}
R_1
=&
\frac{1}{t}\left( \op{W}_{m_1}^{-1}
 \Bigl[ 
  \bigl( \op{W}_{-m_2}\cc{A_{2}^{(\iota)}} \bigr) 
  \bigl( \op{W}_{m_3}{A_{3}^{(\iota)}} \bigr)
 \Bigr]
 - 
 \cc{A_{2}^{(\iota)}} {A_{3}^{(\iota)}}
\right) \\
=&
\frac{1}{t}\left(\op{W}_{m_1}^{-1}-1\right)
\left[
\bigl( \op{W}_{-m_2}\cc{A_{2}^{(\iota)}} \bigr)
 \bigl( \op{W}_{m_3}{A_{3}^{(\iota)}} \bigr) 
\right]\\
&
+
\frac{1}{t}
\bigl\{ 
 \left(\op{W}_{-m_2}-1\right)\cc{A_{2}^{(\iota)}}
\bigr\}
(\op{W}_{m_3}A_{3}^{(\iota)}) 
+
\frac{1}{t}
\cc{A_{2}^{(\iota)}} 
\bigl\{
  \left(\op{W}_{m_3}-1\right) {A_{3}^{(\iota)}} 
\bigr\}.
\end{align*}
Therefore  Lemmas~\ref{W} and \ref{lem3} give us
\begin{align*} 
\|R_1\|_{L^{\infty}} 
&\leq 
\frac{1}{t}\cdot
Ct^{-1/2}\|A_2^{(\iota)}\|_{H^{2}}\|A_3^{(\iota)}\|_{H^{2}}\\
&\leq 
\frac{C}{t^{3/2}} \left(\sum_{j \in \{2,\,3\}}
\|A_j\|_{H^{2,1}}\right)^2\\
&\leq 
\frac{C}{t^{3/2}} \left(\sum_{j \in \{2,\,3\}}
\sum_{|\beta|\leq2}\|\op{J}^{\beta}_{m_j}u_j\|_{H^{1}}\right)^2.
\end{align*}
Next we consider the case of $k\ge 1$ and $|\gamma|=0$. 
Because of the relation $m_1=-m_2+m_3$, 
the binomial formula leads to 
\begin{align*}
 (im_1 \xi)^{\alpha} p_1(\xi;A)
 =&
 (-im_2\xi+im_3\xi)^{\alpha}
 A_{2}^{(\iota)}\cc{A_{3}^{(\iota)}}\\
 =&
 \sum_{\alpha'\le \alpha} \binom{\alpha}{\alpha'} 
(-im_2\xi)^{\alpha'}(im_3\xi)^{\alpha-\alpha'}
\cc{A_{2}^{(\iota)}} {A_{3}^{(\iota)}}\\
 =&
  \sum_{\alpha'\le \alpha}\binom{\alpha}{\alpha'} 
\cc{A_{2}^{(\alpha'+\iota)}} {A_{3}^{(\alpha-\alpha'+\iota)}}.
\end{align*}
On the other hand, the Leibniz formula yields 
\begin{align*}
(im_1\xi)^{\alpha} \op{F}_{m_1}\op{U}_{m_1}^{-1} F_1
=& \op{F}_{m_1}\op{U}_{m_1}^{-1}\pa_x^{\alpha}F_1\\
=& 
\sum_{\alpha'\leq \alpha}\binom{\alpha}{\alpha'} 
\op{F}_{m_1}\op{U}_{m_1}^{-1}
\Bigl[ (\cc{\pa_x^{\alpha'+\iota} u_2}) ({\pa_x^{\alpha-\alpha'+\iota}u_3})
\Bigr]\\
=&
\frac{1}{t}\sum_{\alpha'\leq \alpha}\binom{\alpha}{\alpha'} 
\op{W}^{-1}_{m_1}
\left[
(\op{W}_{-m_2}\cc{A_2^{(\alpha'+\iota)}})
(\op{W}_{m_3}{A_3^{(\alpha-\alpha'+\iota)}})
\right].
\end{align*}
Piecing them together, we have
\begin{align}
&(i m_1 \xi)^\alpha R_1 \nonumber\\
&=
\frac{1}{t}\sum_{{\alpha' \leq \alpha}}\binom{\alpha}{\alpha'} 
\left\{ 
\op{W}^{-1}_{m_1}
 \left[
  \bigl( \op{W}_{-m_2}\cc{A_2^{(\alpha'+\iota)}} \bigr)
  \bigl( \op{W}_{m_3}{A_3^{(\alpha-\alpha'+\iota)}} \bigr)
 \right] 
 -
 \cc{A_2^{(\alpha' +\iota)}}{A_3^{(\alpha - \alpha'+\iota)}}
\right\}.
\label{Relation}
\end{align}
Therefore we can see as before that 
\begin{align*}
|R_1(t,\xi)|
&\le
 \frac{C}{\jb{\xi}^k} \sum_{|\alpha|\le k}
  \bigl| (im_1\xi)^{\alpha}R_1\bigr|
\\
&\le 
\frac{C}{t^{3/2}\jb{\xi}^k}
\sum_{|\alpha|\leq k}\sum_{\alpha' \leq \alpha}
\|A_2^{(\alpha'+\iota)}\|_{H^2}\|A_3^{(\alpha-\alpha'+\iota)}\|_{H^2} 
\\
&\leq 
\frac{C}{t^{3/2}\jb{\xi}^k}
\left(\sum_{j \in \{2,\,3\}}
\sum_{|\beta|\leq 2}
 \|A_j\|_{H^{2,k+1}}
\right)^2
\\
&\leq 
\frac{C}{t^{3/2}\jb{\xi}^k}
\left(\sum_{j \in \{2,\,3\}}
\sum_{|\beta|\leq 2}
\|\op{J}^{\beta}_{m_j}u_j\|_{H^{k+1}}\right)^2. 
\end{align*}
Finally we consider the case of $k\ge 1$ and $|\gamma|\ge 1$. 
From \eqref{Relation} it follows that
\begin{align*}
&\pa_{\xi}^{\beta}\bigl((i m_1\xi)^{\alpha}R_1\bigr) \nonumber\\
&=
\frac{1}{t}\sumc_{{\alpha' \leq \alpha}}
\pa_{\xi}^{\beta}\left\{ 
\op{W}^{-1}_{m_1}
 \left[
  \bigl( \op{W}_{-m_2}\cc{A_2^{(\alpha'+\iota)}} \bigr)
  \bigl( \op{W}_{m_3}{A_3^{(\alpha-\alpha'+\iota)}} \bigr)
 \right] 
 -
 \cc{A_2^{(\alpha' +\iota)}}{A_3^{(\alpha - \alpha'+\iota)}}
\right\}
\nonumber\\
&= 
\frac{1}{t} 
\sumc_{\substack{ \alpha' \leq \alpha \\ \beta' \leq \beta}}
\biggl\{ 
\op{W}^{-1}_{m_1}
 \left[
 \bigl( \op{W}_{-m_2}\cc{\pa_{\xi}^{\beta-\beta'}A_2^{(\alpha'+\iota)}} \bigr)
 \bigl( \op{W}_{m_3}{\pa_{\xi}^{\beta'}A_3^{(\alpha-\alpha'+\iota)}} \bigr) 
\right] 
-
 (\cc{\pa_{\xi}^{\beta-\beta'}A_2^{(\alpha' +\iota)}})
 ({\pa_{\xi}^{\beta'}A_3^{(\alpha - \alpha'+\iota)}})
\biggr\},
\end{align*}
whence
\begin{align*}
|\pa_\xi^\gamma R_1(t,\xi)|
&\leq 
\frac{C}{\jb{\xi}^k} 
\sum_{\beta \leq \gamma} \sum_{|\alpha|\leq k}
|\pa_{\xi}^{\beta}\left( (i m_1 \xi)^\alpha R_1 \right)|\\
&\le 
\frac{C}{t^{3/2}\jb{\xi}^k}
\sum_{\beta \leq \gamma} \sum_{|\alpha|\leq k}
\sum_{\substack{ \alpha' \leq \alpha \\ \beta' \leq \beta}}
\|\pa_{\xi}^{(\beta-\beta')}A_2^{(\alpha'+\iota)}\|_{H^2}
\|\pa_{\xi}^{\beta'}A_3^{(\alpha-\alpha'+\iota)}\|_{H^2} 
\\
&\leq 
\frac{C}{t^{3/2}\jb{\xi}^k}
\left(
\sum_{j \in \{2,\,3\}}
\sum_{|\beta|\leq 2}
 \|A_j\|_{H^{|\gamma|+2,k+1}}
\right)^2
\\
&\leq 
\frac{C}{t^{3/2}\jb{\xi}^k}
\left(\sum_{j \in \{2,\,3\}}
\sum_{|\beta|\leq |\gamma|+2}
\|\op{J}^{\beta}_{m_j}u_j\|_{H^{k+1}}\right)^2. 
\end{align*}
as desired.\qed

\subsection{Derivation of \eqref{est_yayakoshii} }\label{sec:A2}
Let $|\alpha|=p$, $|\beta|=q$. Remember that we assume 
$0\le p+q\le 11$ and $q\le 5$. We may also assume 
$|\rho|\le|\rho'|$  in \eqref{G_key} without loss of generality.
We will divide the argument into three cases.

\noindent\underline{{\bf (i)}\ $q\le 3$:}\ 
Noting the relations
\[
\min\bigl\{|\sigma|,\, |\sigma'|\bigr\}\leq \left[\frac{p+2}{2}\right]\leq \left[\frac{13-q}{2}\right]\leq 8-q,
\]
we deduce from \eqref{G_key} that 
\begin{align*}
 \|G_j^{(\alpha,\beta)}\|_{L^2}
 \le &
 C \sum_{k, l\in I_N}
 \left\{ \| u_k\|_{W^{8,\infty}}
 \|\op{J}_{m_l}^{\beta} u_l\|_{H^{p}}+
 \sum_{1\le |\gamma|\le q}
 \left(
  \|\op{J}_{m_k}^{\gamma} u_k\|_{W^{8-|\gamma|,\infty}}
  \sum_{|\gamma'|\le q-|\gamma|} \|\op{J}_{m_l}^{\gamma'} u_l\|_{H^{p+1}}
 \right)
\right\}\\
 \le &
  C \sum_{|\gamma|\le 3} 
  \sum_{k\in I_N}\|\op{J}_{m_k}^{\gamma} u_k\|_{W^{8-|\gamma|,\infty}}
 \ene{p+\min\{1,\,|\gamma|\}, q-|\gamma|}. 
\end{align*}

\noindent\underline{{\bf (ii)}\ $q=4$:}\ 
First we consider the terms of $(\rho,\rho')=(0,\beta)$ in \eqref{G_key}. 
We use the relations $|\sigma|\le p+1\le 12-q=8$ and $|\sigma'|\le p$ 
as follows:
\begin{align*}
\sumc_{
  \substack{
                  |\sigma|+|\sigma'| \le p+2\\ 
                     \max\bigl\{ 0+|\sigma|,\,  4+|\sigma'|  \bigr\}\leq p+4\\ 
                        \max\bigl\{ |\sigma|,\,  |\sigma'|  \bigr\}\leq p+1
               }
           }  
 \bigl\|
  \sh{(\pa_x^{\sigma} u_k)}\sh{(\pa_x^{\sigma'}\op{J}_{m_l}^{\beta} u_l)}
 \bigr\|_{L^2}
 \le 
 C \|u_k\|_{W^{8,\infty}} \|\op{J}_{m_l}^{\beta} u_l\|_{H^{p}}
 \le 
 C \|u_k\|_{W^{8,\infty}}  \ene{p,4}.
\end{align*}
As for the other terms, it follows from the relations  
\begin{align}
\min\bigl\{ |\sigma|,\,  |\sigma'|  \bigr\}\leq \left[\frac{p+2}{2}\right]\leq \left[\frac{13-q}{2}\right]=4\leq 8-\max\bigl\{|\rho|,\, |\rho'|\bigr\}
\label{index_sigma}
\end{align}
and $|\rho|\le|\rho'|\le 3$ that
\begin{align*}
\sum_{k,l\in \sh{I_N}}
\sum_{  \substack{
                   |\rho|+|\rho'| \le 4\\ 
                    |\rho|\le|\rho'|\le3\\
                           }
          } 
 \sumc_{  \substack{
                  |\sigma|+|\sigma'| \le p+2\\ 
           \max\bigl\{ |\rho|+|\sigma|,\,  |\rho'|+|\sigma'|  \bigr\}\leq p+4\\ 
                        \max\bigl\{ |\sigma|,\,  |\sigma'|  \bigr\}\leq p+1
                           }
          }&
 \bigl\|
  \sh{(\pa_x^{\sigma}\op{J}_{m_k}^{\rho} u_k)}
  \sh{(\pa_x^{\sigma'}\op{J}_{m_l}^{\rho'} u_l)}
 \bigr\|_{L^2}\\
 \le &
 C\sum_{|\gamma|\le 3}\sum_{k,l\in I_N}
 \left(
  \|\op{J}_{m_k}^{\gamma} u_k\|_{W^{8-|\gamma|,\infty}}
  \sum_{|\gamma'|\le 4-|\gamma|} \|\op{J}_{m_l}^{\gamma'} u_l\|_{H^{p+1}}
 \right)\\
 \le &
  C \sum_{|\gamma|\le 3}
  \sum_{k\in I_N} 
  \|\op{J}_{m_k}^{\gamma} u_k\|_{W^{8-|\gamma|,\infty}}
 \ene{p+1, 4-|\gamma|}. 
\end{align*}
Summing up, we obtain the desired inequality for $q=4$. 


\noindent\underline{{\bf (iii)}\ $q=5$:}\ 
Since $p+1\le 12-q=7$, we see as before that 
\begin{align*}
 \sumc_{  \substack{
                  |\sigma|+|\sigma'| \le p+2\\ 
                     \max\bigl\{ 0+|\sigma|,\,  5+|\sigma'|  \bigr\}\leq p+5\\ 
                        \max\bigl\{ |\sigma|,\,  |\sigma'|  \bigr\}\leq p+1
               }} 
 \bigl\|
  \sh{(\pa_x^{\sigma} u_k)}\sh{(\pa_x^{\sigma'}\op{J}_{m_l}^{\beta} u_l)}
 \bigr\|_{L^2}
 \le 
 C \|u_k\|_{W^{7,\infty}}  \ene{p,5}
\end{align*}
and
\begin{align*}
 \sum_{k,l\in\sh{I_N}}
 \sum_{\substack{ |\rho|\le 1\\ |\rho'|=4}}
 \sumc_{  \substack{
                  |\sigma|+|\sigma'| \le p+2\\ 
                  \max\bigl\{|\rho|+|\sigma|,\,  4+|\sigma'|  \bigr\}\leq p+5\\ 
                        \max\bigl\{ |\sigma|,\,  |\sigma'|  \bigr\}\leq p+1
               }} &
 \bigl\|
  \sh{(\pa_x^{\sigma}\op{J}_{m_k}^{\rho}  u_k)}
  \sh{(\pa_x^{\sigma'}\op{J}_{m_l}^{\rho'} u_l)}
 \bigr\|_{L^2}\\
 \le &
 C \sum_{|\gamma|\le1}\sum_{k\in I_N}
\|\op{J}_{m_k}^{\gamma}u_k\|_{W^{8-|\gamma|,\infty}} \ene{p+1,5-|\gamma|}.
\end{align*}
As for the other terms, 
we deduce from \eqref{index_sigma} that 
\begin{align*}
 \sum_{k,l\in\sh{I_N}}
 \sum_{\substack{ |\rho|\le 2\\ |\rho'|\le 3}}
 \sumc_{  \substack{
                  |\sigma|+|\sigma'| \le p+2\\ 
           \max\bigl\{ |\rho|+|\sigma|,\, |\rho'|+|\sigma'|  \bigr\}\leq p+5\\ 
                        \max\bigl\{ |\sigma|,\,  |\sigma'|  \bigr\}\leq p+1 
               }} &\
 \bigl\|
  \sh{(\pa_x^{\sigma}\op{J}_{m_k}^{\rho}  u_k)}
  \sh{(\pa_x^{\sigma'}\op{J}_{m_l}^{\rho'} u_l)}
 \bigr\|_{L^2}\\
 \le &
 C \sum_{|\gamma|\le 3}\sum_{k\in I_N}
\|\op{J}_{m_k}^{\gamma}u_k\|_{W^{8-|\gamma|,\infty}} \ene{p+1,5-|\gamma|}.
\end{align*}
Piecing them all together, we arrive at the desired inequality for $q=5$. 

\qed

\medskip
\subsection*{Acknowledgments}
The authors thank Professor Soichiro Katayama for his useful 
conversations on this subject. 
The work of H.~S. is supported by Grant-in-Aid for Scientific Research (C) 
(No.~17K05322), JSPS.


\end{document}